\documentclass[10pt]{amsart}

\usepackage[all]{xy}
\usepackage{amsmath}
\usepackage{amsfonts}
\usepackage{amssymb}
\usepackage{amscd}
\usepackage{amsthm}
\usepackage{latexsym}
\usepackage{amsbsy}
\usepackage{graphicx}
\usepackage{slashed}
\usepackage{epstopdf}
\newtheorem{lem}{Lemma}[section]
\newtheorem{prop}{Proposition}[section]

\newtheorem{cor}{Corollary}[section]
\newtheorem{defi}{Definition}[section]
\newtheorem{exm}{Example}[section]

\newtheorem{rmk}{Remark}[section]

\usepackage{lipsum}
\def \tr{{\rm tr}}
\def \Tr{{\rm Tr}}
\def \TR{{\rm TR}}
\def \Res{{\rm Res}}
\def \res{{\rm res}}
\def \A{\mathcal{A}^n_{\theta}}
\def\Int{\int\hspace{-0.35cm}{-} \,}
\def\Wres{{\rm Wres}}
\def\sgn{{\rm sgn}}
\def\SF{{\rm SF}}
\def\dvol{{\rm dvol}}
\def\index{{\rm index}}
\def\f.p.{{\rm f.p.}}
\def\spec{{\rm spec}}
\def\bos{{\rm bos}}
\def\fer{{\rm fer}}
\def\ord{{\rm ord}}
\def\T^3{{\rm C^{\infty}(\mathbb{T}^3_{\theta})}}
\title{\bf On Certain Spectral Invariants of  Dirac Operators on Noncommutative Tori}

\author{ $ $\\ Ali Fathi, Masoud Khalkhali} 
\date{}

\begin{document}

\maketitle

\begin{center} 

Department of Mathematics, The University of Western Ontario\\London, ON, Canada\footnote{{\it E-mail addresses}:  
afathiba@uwo.ca, masoud@uwo.ca}
 \\

\end{center}

\begin{abstract}
  The spectral eta function for certain families of Dirac operators on  noncommutative $3$-torus is considered and the regularity at zero is proved.  By using  variational techniques, we show that $\eta_{D}(0)$ is a conformal invariant. By studying the Laurent expansion at zero of $\text{TR} (|D|^{-z})$, the conformal invariance of $\zeta'_{|D|}(0)$ for noncommutative $3$-torus is proved. Finally, for the coupled Dirac operator, a local formula for the variation $\partial_A\eta_{D+A}(0)$  is derived which is the  analogue of  the  so called induced Chern-Simons term in quantum field theory literature.  

\end{abstract}

\tableofcontents

\section{Introduction}

In this paper  we study  the variations of spectral zeta and eta functions  
$\zeta_{|D|} (z)= \TR(|D|^{-z})$ and $\eta_D (z)=\text{TR} (D|D|^{-z-1}) $  associated to  certain families of  Dirac operators on noncommutative $3$-torus.  In the classical case  the canonical trace TR  \cite{Kontsevich-Vishik1995} provides a unified method of studying various spectral functions of elliptic operators and their variations.  Connes'  pseudodifferential calculus for  noncommutative tori  makes it possible to define a suitable notion of noncommutative canonical trace  \cite{det-work}, and translate some of the  properties of the canonical trace on manifolds to noncommutative settings.  Among these, a  fundamental result is the explicit description of the Laurent expansion at zero of the function $\text{TR} (AQ^{-z})$ where $A$ and $Q$ are classical elliptic operators \cite{Paycha-Rosenberg2006}. This result enables us to prove the regularity of $\zeta_{|D|}(z)$ and $\eta_D(z)$ at $z=0$,  and also gives a local description for variations of $\eta_D(0)$ and $\zeta'_{|D|}(0)$. In particular, we show that $\eta_D(0)$ is constant over the family $\{e^{th}De^{th}\}$ and hence is a conformal invariant of noncommutative $3$-torus. Also, using the local description for conformal variation of $\zeta'_{|D|}(0)$,  we prove that this quantity is a conformal invariant of noncommutative $3$-torus.
This paper is organized as follows. In Section 2  we recall the definition of a spectral triple which is the  basic ingredient in the definition of a noncommutative Riemannian space \cite{Connesbook1994}. Our main example is the spin spectral triple for  noncommutative tori  and its conformal perturbation first proposed in 
\cite{Cohen-Connes1992, Connes-Tred2009}.  In Section 3 we give a brief review  of Connes' pseudodifferential  calculus  for noncommutative tori from \cite{Connes1980, Connes-Tred2009}, and recall 
 the extension of  the Kontsevich-Vishik  canonical trace  to the setting of noncommutative tori from \cite{det-work}.  It should be mentioned that  this is also done in \cite{Paycha-Levy2014}   where  one works   with toroidal symbols instead  of Connes' symbols.
 
 In Section 4 we study the eta function associated to the Dirac operators of the conformally perturbed  spectral triples  $(\T^3,\mathcal{H},e^{th}De^{th}),$  and also to the coupled Dirac operator of the spectral triple $(\T^3,\mathcal{H},D+A)$. By exploiting the developed  canonical trace, the regularity of the eta function at zero in above cases will be proved. Next, by using variational techniques we show that the value of the eta function at zero is constant over a conformally perturbed family.
Also, by considering the spectral triple  $(\T^3,\mathcal{H},D$) and the family $D_t=D+tu^*[D,u]$ for a unitary element $u\in\T^3$, we relate the difference $\eta_{D_1}(0)-\eta_{D_0}(0)$ to the spectral flow of the family $D_t$ and  give a local formula for index of the operator $PuP$. This is the analogue of the result of Getzler \cite{Getzler-flow},  in the case of noncommutative $3$-torus. 

  In Section 5  we consider the spectral zeta function $\zeta_{|D|}(z)=\TR(|D|^{-z})$ and study the conformal variation of the spectral value $\zeta'_{|D|}(0)$ within the framework of the canonical trace.
  We show that for the noncommutative $3$-torus this quantity is a conformal invariant. In even dimensions though, the conformal variation is not zero and hence conformal anomaly exists. Following \cite{Connes-Moscovici2014} we give a local formula for the conformal variation of $\zeta_{\Delta}'(0)$ in the case of noncommutative two torus. 
 
  Finally, in section 6 we consider the coupled Dirac operator $D+A$ and study the value $\zeta'_{D}(0)$ where $\zeta_D(z)=\TR(D^{-z})$. Since the spectrum of $D$ is extended along the real line, there is an ambiguity in the definition of the complex power $D^{-z}$ and hence in the value $\zeta'_D(0)$. In odd dimensions, this ambiguity can be expressed in terms of $\eta_{D+A}(0)$ and hence has a dependence on the coupled gauge field $A$. This dependence can be computed by a local formula and in  physics literature it is usually referred to as the induced  Chern-Simons term generated by the coupling of a massless fermion to a classical gauge field (cf. e.g. \cite{QFT-analytic}). We give an analogue of this computation and the local formula in the case of noncommutative $3$- torus.

Conformal  and complex geometry of    noncommutative two tori were first   studied in   the seminal  work  of Connes and Tretkoff \cite{Connes-Tred2009}   where a Gauss-Bonnet theorem was  proved for a  conformally perturbed  metric (cf.  \cite{Cohen-Connes1992}
for a preliminary version). 
This result was  extended in  \cite{Masoud-Farzad2012} where the Gauss-Bonnet theorem was  proved  for  metrics in all translation invariant conformal structures. 
 The  problem of  computing the scalar curvature of  the curved noncommutative two 
torus was fully settled  in \cite{Connes-Moscovici2014}, and   in \cite{Farzad-Masoud2013},  and in \cite{Farzad-Masoud2014, FF1} in  the four dimensional case.   Other  works on  the spectral geometry and heat kernel invariants of curved noncommutative tori  include
\cite{Marcolli-Tanvir2012, Ali-Masoud2014, dabsit1,  dabsit2, FF1}.   The computation of the curvature of the determinant line bundle in the sense of Quillen for certain families of   Dirac operators on noncommutative tori  was carried out in \cite{det-work}.

We would like to thank  Asghar Ghorbanpour for useful  discussions on the subject of this paper. MK  would like to thank 
the Hausdorff Institute in Bonn for its hospitality and support while  this work was being  completed.

\section{Noncommutative  geometry framework }
In this section we recall the basic ingredients in the definition of a noncommutative geometry. Our main example is the spin spectral triple for  noncommutative tori on which all the material in this paper is based. 
The data of a noncommutative Riemannian geometry is encoded in  a spectral triple \cite{Connes1980}.

\begin{defi}
  A spectral triple $(\mathcal{A},\mathcal{H},D)$ is given by an involutive unital (possibly noncommutative) algebra $\mathcal{A}$, a representation $\pi:\mathcal{A}\longrightarrow\mathcal{B}(\mathcal{H})$ on  a Hilbert space $\mathcal{H}$,  and a self-adjoint densely defined operator $D: Dom (D) \subset \mathcal{H} \longrightarrow\mathcal{H}$ with
compact resolvent and the property that $[D, \pi(a)]$ is bounded for any $a\in\mathcal{A}$.
\end{defi}
A spectral triple $(\mathcal{A},\mathcal{H},D)$ is called {\it even}  if there exist a self-adjoint unitary operator $\Gamma:\mathcal{H}\to\mathcal{H}$,  such that
$a\Gamma=\Gamma a,$
for $a\in\mathcal{A}$ 
and 
$D\Gamma=-\Gamma D.$ 
The operator $\Gamma$ induces a grading $\mathcal{H}=\mathcal{H}^+\oplus\mathcal{H}^-$ with respect to which the Dirac operator is odd,
$$D=
\begin{bmatrix}
0&D^-\\
D^+&0
\end{bmatrix}.$$

It can be shown that the triple $(C^{\infty}(M), L^2(M,S), \slashed{D})$, consisting of $\mathcal{A}=C^{\infty}(M)$, the  algebra of smooth functions on a closed  Riemannian spin manifold $(M,g)$,  and the spin Dirac operator $\slashed{D}$ on the Hilbert space  $\mathcal{H}= L^2(M,S)$ of $ L^2$-spinors satisfies the requirements of a spectral triple (cf. e.g. \cite{Green-book'}). In even dimensions, the spinor bundle admits a grading and we have $S=S^{+}\oplus S^{-}$ with respect to which the Dirac operator is odd. Therefore in even dimensions the spectral triple $(C^{\infty}(M), L^2(M,S), \slashed{D})$ is even.

\subsection{Geometry of noncommutative tori}

 Let  $\Theta\in M_n(\mathbb{R})$ be a skew symmetric matrix.  The noncommutative $n$-torus $C(\mathbb{T}^n _{\Theta})$
is defined  to be the universal $C^{*}$-algebra generated by the unitaries $U_k$ for $k\in\mathbb{Z}^n$ and relations,
 $$U_kU_l=e^{\pi i\Theta(k,l)}U_{k+l},\quad k,l\in\mathbb{Z}^n.$$
  Consider the standard basis $\left\{e_i\right\}$ for $\mathbb{R}^n$ and let $u_i=U_{e_i}$. Then it follows that 
  $$u_ku_l=e^{2\pi i\theta_{kl}}u_lu_k,$$
  where $\theta_{kl}=\Theta(e_k,e_l)$.
 The smooth noncommutative $n$-torus $C^{\infty}(\mathbb{T}^n_{\Theta})$  is defined to be  the Fr\'{e}chet $*-$subalgebra of elements with Schwartz coefficients in the Fourier expansion, that is all  
 the $a\in C^{\infty}(\mathbb{T}^n_{\theta})$ that can be written as
 $$a=\sum_{k\in\mathbb{Z}^n}a_k U_k,$$ where $\{a_k\}\in\mathcal{S}(\mathbb{Z}^n)$.
In fact, $C^{\infty}(\mathbb{T}^n_{\theta})$ is a deformation of $C^{\infty}(\mathbb{T}^n)$ and consists of the smooth vectors under the periodic action of $\mathbb{R}^n$ on $C(\mathbb{T}^n _{\Theta})$ given by
$$\alpha_s(U_k)=e^{is.k}U_k,~~s\in\mathbb{R}^n, k\in\mathbb{Z}^n.$$
The algebra $C^{\infty}(\mathbb{T}^n_{\theta})$ is equipped with a tracial state  given by 
$$\tau(\sum_{p\in\mathbb{Z}^n}a_pU_p)=a_0.$$
We also denote by $\delta_{\mu}$ the analogues of the partial derivatives $ \frac{1}{i}\frac{\partial}{\partial x^{\mu}}$ on $C^{\infty}(\mathbb{T}^n)$ which are derivations on the algebra $C^{\infty}(\mathbb{T}^n_{\theta})$ defined by
$$\delta_{\mu}(U_k)=k_{\mu} U_k.$$
These derivations have the following property
$$\delta_{\mu}(a^*)=-(\delta_{\mu}a)^*,$$
and also satisfy  the integration by parts formula
$$\tau(a\delta_{\mu}b)=-\tau((\delta_{\mu}a)b),\quad a,b\in C^{\infty}(\mathbb{T}^n_{\theta}).$$

By GNS construction one gets the Hilbert space $\mathcal{H}_{\tau}$ on which $C^{\infty}(\mathbb{T}^n_{\theta})$ is represented by left multiplication which will be denoted by $\pi(a)$ for $a\in C^{\infty}(\mathbb{T}^n_{\theta})$.
The spectral triple describing the noncommutative geometry of noncommutative $n$-torus consists of the algebra $C^{\infty}(\mathbb{T}^n_{\theta})$ , the Hilbert space $\mathcal{H}=\mathcal{H}_{\tau}\otimes\mathbb{C}^{N}$, where $N=2^{[n/2]}$ with the inner product on $\mathcal{H}_{\tau}$  given by $$\left<a,b\right>_{\tau}=\tau(b^*a),$$ 
and the representation of $C^{\infty}(\mathbb{T}^n_{\theta})$ given by $\pi\otimes 1$.

 The Dirac operator is 
$$D=\slashed{\partial}:=\partial_{\mu}\otimes\gamma^{\mu},$$
where $\partial_{\mu}=\delta_{\mu}$, seen as an unbounded self-adjoint operator on $\mathcal{H}_{\tau}$ and $\gamma^{\mu}$s are Clifford (Gamma) matrices in $M_N(\mathbb{C})$ satisfying the relation
$$\gamma^i\gamma^j+\gamma^j\gamma^i=2\delta^{ij} I_N.$$
In $3-$dimension the Clifford matrices are given by the Pauli spin matrices,
$$\gamma^{1}=
\begin{bmatrix}
0&1\\
1&0
\end{bmatrix} ,
\gamma^{2}=
\begin{bmatrix}
0&-i\\
i&0
\end{bmatrix},
\gamma^{3}=
\begin{bmatrix}
1&0\\
0&-1
\end{bmatrix}.
$$

Consider the chirality matrix $$\gamma=(-i)^m\gamma^1\cdots\gamma^n,$$ where $n=2m$ or $2m+1$. It is seen that $\gamma\cdot\gamma=1$,  and $\gamma$ anti-commutes with every $\gamma^{\mu}$ for $n$ even. The operator $\Gamma=1\otimes\gamma$ on $\mathcal{H}=\mathcal{H}_{\tau}\otimes\mathbb{C}^{2^{[n/2]}}$ defines a grading on $\mathcal{H}$ and for even $n$ one has
$$\Gamma D=-D\Gamma.$$
Therefore $(C^{\infty}(\mathbb{T}^n_{\theta}),\mathcal{H},D)$ is an even spectral triple for $n$ even.

\subsection{Real structure }

\begin{defi}
\normalfont
A real structure on a spectral triple  $(\mathcal{A},\mathcal{H},D)$ is an anti-unitary operator  $J:\mathcal{H}\to\mathcal{H}$ such   $ J (Dom D) \subset Dom D$, and   $[ a,  Jb^*J^{-1}]=0$ for all  $a, b\in\mathcal{A}$. We say $(\mathcal{A},\mathcal{H},D)$ is of KO-dimension $n$ if the operator $J$ satisfies the following commutation relations 
$$J^2=\epsilon 1,~JD=\epsilon' DJ,~J\Gamma=\epsilon''\Gamma J,$$
where the $\epsilon,\epsilon',\epsilon''$ depend on $n\in\mathbb{Z}_8$ according to the following table:
\label{real-table}
$$
\begin{tabular}{|l|l|l|l|l|l|l|l|l|}
  \hline
  n&0&1&2&3&4&5&6&7 \\
  \hline
  $\epsilon$&1&1&-1&-1&-1&-1&1&1 \\
  \hline
   $\epsilon'$&1&-1&1&1&1&-1&1&1 \\
  \hline 
  $\epsilon''$&1&&-1&&1&&-1&\\
  \hline
\end{tabular}
$$
\end{defi}
Note that the real structure operator switches left action to right action. More precisely, Since $\mathcal{A}$ commutes with $J\mathcal{A}J^*$ the real structure gives a representation of the opposite algebra $\mathcal{A}^{op}$ by $b^{\circ}=Jb^*J^*$ and turns the Hilbert space $\mathcal{H}$ into an $\mathcal{A}-$bimodule by
$$a\psi b=aJb^*J^*(\psi)~~a,b\in\mathcal{A},~\psi\in\mathcal{H}.$$
\begin{exm}\label{real-tori}
\normalfont
The spectral triples $(C^{\infty}(\mathbb{T}^n_{\theta}),\mathcal{H}=H_{\tau}\otimes\mathbb{C}^N,D)$ for $N=2^{[n/2]}$ are all equipped with real structure.
The real structure operator $J:H_{\tau}\otimes\mathbb{C}^N\rightarrow H_{\tau}\otimes\mathbb{C}^N$ is given by
$$J=J_0\otimes C_0,$$
where $J_0$ is the Tomita conjugation map associated to the GNS Hilbert space $H_{\tau}$ given by 
$$J_0(a)=a^*,~a\in C^{\infty}(\mathbb{T}^n_{\theta})$$
and $C_0$ is the charge conjugation operator on $\mathbb{C}^N$ (cf. e.g. \cite{Green-book'}).  For $N=2$ we have
$$J=
\begin{bmatrix}
0&-J_0\\
J_0&0
\end{bmatrix}.$$
$J_0^2=1$ for all $n$-mod 8, so by extending the Hilbert space $H_{\tau}$ and coupling  $J_0$ with $C_0$ one can check that the resulting operator $J$ satisfies the requirements for a real structure (see \cite{Green-book'}). 
\end{exm}
\subsection{The coupled Dirac operator}

In this section We  consider the spectral triples obtained by coupling (twisting) the Dirac operator by a gauge potential. We begin by recalling the construction in commutative case.
 
 Consider a compact Riemannian spin manifold $(M,g)$ with the spin Dirac operator $D:L^2(S)\rightarrow L^2(S)$ on spinors. Let $V\rightarrow M$ be a Hermitian vector bundle equipped with a compatible connection $\nabla^V$. Then one constructs the coupled (twisted) spinor bundle $S\otimes V\rightarrow M$ with the extended Clifford action
\begin{equation*}
\pi(\omega).(\psi\otimes f)= (\omega.\psi)\otimes f,
\end{equation*}
where $\pi(\omega)$ is a section of the Clifford bundle over $M$ represented on the space of spinors and $\psi$ and $f$ are sections of spinor bundle and the vector bundle $V$ respectively. 
Also, the vector bundle $S\otimes V$ is equipped with the coupled (twisted) spin connection
\begin{equation*}
 \nabla^{S\otimes V}=\nabla^S\otimes1+1\otimes\nabla^V,
\end{equation*}
 where $\nabla^S$ is the spin connection on $S$. This connection gives rise to the coupled Dirac operator $D_{\nabla^V}$ acting on the sections of $S\otimes V$. When $V$ is a trivial vector bundle, the connection $\nabla^V$ can be globally written as $\nabla^V=d+A$ where $A$ is a matrix of one forms (vector potential) and an easy computation shows that
 \begin{equation*}
 D_A=D+\pi(A).
 \end{equation*}
 
The above construction can be generalized to the setting of spectral triples. Starting with a spectral triple $(\mathcal{A},\mathcal{H},D)$, one can construct a new spectral triple $(\mathcal{A},\mathcal{H},D+A)$ by adding a gauge potential  to the Dirac operator, this corresponds to picking a trivial projective module over the algebra $\mathcal{A}$.

 More precisely, $A$ is a self-adjoint element of 
$$\Omega^1_{D}=\left\{\sum_j a_j^0[D,a_j^1],~a_j^i\in\mathcal{A}\right\}.$$
In fact, $\Omega^1 _{D}$ is the image inside $\mathcal{B}(\mathcal{H})$ of the noncommutative 1-forms on $\mathcal{A}$ under the induced map 
$$\pi: a^0da^1\longrightarrow a^0[D,a^1].$$

Below, we explicitly write down the coupled Dirac operator for the spectral triple $(C^{\infty}(\mathbb{T}^n_{\theta}),\mathcal{H},D+A)$.
First note that for any element $a=\sum_{k\in\mathbb{Z}^n}a_kU_k$ in $C^{\infty}(\mathbb{T}^n_{\theta})$ we have,
$$[D,a]=\partial_{\mu}(a)\otimes\gamma^{\mu}.$$
Any $A\in\Omega^1(C^{\infty}(\mathbb{T}^n_{\theta}))$ is of the form $A=\sum_ia_idb_i$ where $a_i,b_i$ are in $C^{\infty}(\mathbb{T}^n_{\theta})$ and we have
$$\pi(A)=\sum_ia_i\partial_{\mu}(b_i)\otimes\gamma^{\mu}.$$
We denote the elements $ \sum_ia_i\partial_{\mu}b_i$ by $A_{\mu}$ and hence,
$$\pi(A)=A_{\mu}\otimes\gamma^{\mu},$$
also self adjointness of $A$ gives $A_{\mu}^*=A_{\mu}$.
Therefore, the coupled Dirac operator is given by
$$D+A=\slashed{\partial}+\slashed{A},$$
 where again by Feynman slash notation $\slashed{A}=A_{\mu}\otimes\gamma^{\mu}$.
\section{Elliptic theory on noncommutative tori}
 Our aim in this section is to recall the extension of  the Kontsevich-Vishik  canonical trace  to the setting of noncommutative tori from \cite{det-work}.  Alternatively, this is also done in \cite{Paycha-Levy2014}   where  they work  with toroidal symbols instead  of Connes' symbols.
We begin by a brief review of the basics of Connes' pseudodifferential  calculus  for noncommutative tori from \cite{Connes1980, Connes-Tred2009}.

\subsection{Matrix pseudodifferential calculus on $C^{\infty}(\mathbb{T}^n_{\theta})$}

We shall use the multi-index notation $\alpha=(\alpha_1,..,\alpha_n)$, $\alpha_i\geq0$ , $|\alpha|=\alpha_1+...+\alpha_n$, $\alpha!=\alpha_1!...\alpha_n!$, $\delta^{\alpha}=\delta_1^{\alpha_1}...\delta_n^{\alpha_n}$ and $\partial_{\xi}^{\beta}=
\partial_{\xi_1}^{\beta_1}...\partial_{\xi_n}^{\beta_n}$.
\begin{defi}
A matrix valued symbol of order $m$ on noncommutative $n-$torus is a smooth map 
$$\sigma:\mathbb{R}^n \longrightarrow C^{\infty}(\mathbb{T}^n_{\theta})\otimes M_N(\mathbb{C}),$$
such that
$$||\delta^{\alpha}\partial_{\xi}^{\beta}\sigma(\xi)||\leq C_{\alpha, \beta} (1+|\xi|)^{m-|\beta|},$$
and there exists a smooth map $k:\mathbb{R}^n \setminus \{0\}\longrightarrow C^{\infty}(\mathbb{T}^n_{\theta})\otimes M_N(\mathbb{C})$ such that 
$$ \lim_{\lambda\rightarrow\infty}\lambda^{-m}\sigma(\lambda\xi_1,\lambda\xi_2,...,\lambda\xi_n)=k(\xi_1,\xi_2,...,\xi_n).$$
We denote the symbols of order $m$ by $\mathcal{S}^m(C^{\infty}(\mathbb{T}^n_{\theta}))$.
\end{defi}

A matrix pseudodifferential operator associated with $\sigma\in\mathcal{S}^m(C^{\infty}(\mathbb{T}^n_{\theta}))$ is the operator $A_{\sigma}:C^{\infty}(\mathbb{T}^n_{\theta})\otimes\mathbb{C}^N
\longrightarrow C^{\infty}(\mathbb{T}^n_{\theta})\otimes\mathbb{C}^N$ defined by 
$$A_{\sigma}(a)=\int_{\mathbb{R}^n}\int_{\mathbb{R}^n} e^{-is.\xi}\sigma(\xi)\alpha_s(a)dsd\xi,$$
where $\alpha_s$ is the extended action of $\mathbb{R}^n$ on $C^{\infty}(\mathbb{T}^n_{\theta})\otimes\mathbb{C}^N$.

Two symbols $\sigma$, $\sigma'\in {\mathcal S}^m(C^{\infty}(\mathbb{T}^n_{\theta}))$  are considered equivalent if  $\sigma-\sigma'\in {\mathcal S}^m(C^{\infty}(\mathbb{T}^n_{\theta}))$ for all  $m$. The equivalence of the symbols will be denoted by  $\sigma \sim \sigma'$.
We denote the collection of pseudodifferential operators by $\Psi^*(C^{\infty}(\mathbb{T}^n_{\theta}))$. The order gives a natural filtration on $\Psi^*(C^{\infty}(\mathbb{T}^n_{\theta}))$ and the following proposition \cite{Connes-Tred2009} gives an explicit formula for the symbol of the  product  of  pseudodifferential operators as operators on $\mathcal{H}$  modulo the above equivalence relation.
\begin{prop}
Let $P$ and $Q$ be pseudodifferential operators with the symbols $\sigma$ and $\sigma'$ respectively. The product $PQ$ is a pseudodifferential operator with the following symbol,
$$\sigma(PQ)\sim\sum_{\alpha}\frac{1}{\alpha!}\partial_{\xi}^{\alpha}(\sigma(\xi))\delta^{\alpha}(\sigma'(\xi)).$$
\end{prop}\qed

\begin{defi}

A symbol $\sigma\in\mathcal{S}^m(C^{\infty}(\mathbb{T}^n_{\theta}))$ is called elliptic if $\sigma(\xi)$ is invertible for $\xi\neq0,$ and for some $c$
$$||\sigma(\xi)^{-1}||\leq c(1+|\xi|)^{-m},$$ for large enough $|\xi|$.
\end{defi}

\begin{exm}
\normalfont
Consider the Dirac operator of the spectral triple $\left(\T^3,\mathcal{H}_{\tau}\otimes\mathbb{C}^2,D=\slashed{\partial}\right)$,
$$\slashed{\partial}=\partial_{\mu}\otimes\gamma^{\mu}:\mathcal{H}_{\tau}\oplus\mathcal{H}_{\tau}\longrightarrow\mathcal{H}_{\tau}\oplus\mathcal{H}_{\tau}.$$
The symbol reads
$$\sigma(D)(\xi)=\xi_{\mu}\otimes\gamma^{\mu}=
\begin{bmatrix}
\xi_3&\xi_1-i\xi_2\\
\xi_1+i\xi_2&-\xi_3
\end{bmatrix},$$
which is clearly elliptic.
\end{exm}

A  smooth map $\sigma:\mathbb{R}^n\to C^{\infty}(\mathbb{T}^n_{\theta})\otimes M_N(\mathbb{C})$ is called a classical symbol of order $\alpha \in \mathbb{C}$  
if for any $L$ and each $0\leq j\leq L$ there exist $\sigma_{\alpha-j}:\mathbb{R}^n\backslash\{0\}\to C^{\infty}(\mathbb{T}^n_{\theta})\otimes M_N(\mathbb{C})$ positive homogeneous of degree $\alpha-j$,  and a symbol $\sigma^L\in\mathcal{S}^{\Re(\alpha)-L-1}(C^{\infty}(\mathbb{T}^n_{\theta}))$, such that
 \begin{equation}
 \label{sigma}
\sigma (\xi)=\sum_{j=0}^{L}\chi(\xi)\sigma_{\alpha-j}(\xi)+\sigma^L(\xi)\quad\xi\in\mathbb{R}^n.
 \end{equation}
 Here  $\chi$ is a smooth cut off function on $\mathbb{R}^n$ which is equal to zero on a small ball around the origin,  and is equal to one outside the unit ball.
  The homogeneous terms in the expansion are uniquely determined by $\sigma$. 
 The set of classical symbols of order $\alpha$ on noncommutative n-torus will be denoted by $\mathcal{S}^{\alpha}_{cl}(C^{\infty}(\mathbb{T}^n_{\theta}))$.

The analogue of the Wodzicki residue for classical pseudodifferential operators on the noncommutative $n$-torus is defined in \cite{Farzad-Wong2011}.
 \begin{defi}
 The noncommutative residue of a classical pseudodifferential operator $A_{\sigma}$ is defined as
 $$\Wres(A_{\sigma})=\tau \left(\res (A_\sigma)\right),$$
where $\res(A_\sigma):=\int_{|\xi|=1} \tr\, \sigma_{-n}(\xi)d\xi$. 
 \end{defi}
It is evident from the definition that noncommutative residue vanishes on differential operators,  operators of order $<-n$ as well as non-integer order operators

 \subsection{The canonical trace} 
In what follows, we recall  the analogue of Kontsevich-Vishik canonical trace \cite{Kontsevich-Vishik1995}  on non-integer order pseudodifferential operators on the noncommutative tori from \cite{det-work}. For an alternative approach based on toroidal noncommutative symbols see \cite{Paycha-Levy2014}. For a thorough review of the theory in the classical case we refer to \cite{Paycha2012,Paycha-Scott2007}.

The existence of the so called cut-off integral for classical noncommutative symbols is established in \cite{det-work}.
 
\begin{prop}\label{cut-off}
 Let $\sigma\in\mathcal{S}_{cl}^{\alpha}(C^{\infty}(\mathbb{T}^n_{\theta}))$ and $B(R)$ be the ball of radius $R$ around the origin. One has the following asymptotic expansion
 $$\int_{B(R)}\sigma(\xi)d\xi\sim_{R\rightarrow\infty}\sum_{j=0,\alpha-j+n\neq0}^{\infty}\alpha_j(\sigma)R^{\alpha-j+n}+\beta(\sigma)\log R+c(\sigma),$$
 where $\beta(\sigma)=\int_{|\xi|=1}\sigma_{-n}(\xi)d\xi$
 and the constant term in the expansion, $c(\sigma)$, is given by 
\begin{equation}
\int_{\mathbb{R}^n}\sigma^L+\sum_{j=0}^{L}\int_{B(1)}\chi(\xi)\sigma_{\alpha-j}(\xi)d\xi-\sum_{j=0,\alpha-j+n\neq0}^L \frac{1}{\alpha-j+n}\int_{|\xi|=1}\sigma_{\alpha-j}(\omega)d\omega.
\end{equation} 
 Here we have used the notation of (\ref{sigma}).
  \end{prop}
 
\begin{defi} 
 The cut-off integral of a symbol $\sigma\in\mathcal{S}_{cl}^{\alpha}(C^{\infty}(\mathbb{T}^n_{\theta}))$ is defined to be the constant term in the above asymptotic expansion, and we denote it by $\Int \sigma(\xi)d\xi$.
\end{defi}

    Now  the canonical trace of a classical pseudodifferential operator of non-integer order on $C^{\infty}(\mathbb{T}^n_{\theta})$ is defined as follows \cite{det-work}:

\begin{defi}
 The canonical trace of a classical pseudodifferential operator $A $ of non-integral order $\alpha$ is defined as 
 \begin{equation*}
 \TR(A) =\tau \left(\Int \tr \, \sigma_A(\xi)d\xi\right).
 \end{equation*}
 \end{defi}

 The relation between the TR-functional and the usual trace on trace-class pseudodifferential operators is established in \cite{det-work}.
Note that any pseudodifferential operator $A$  of order less that  $-n$ is a trace class operator  and its trace is given by
 $$\Tr(A)=\tau \left(\int_{\mathbb{R}^n}\tr\sigma_{P}(\xi)d\xi\right).$$
 The symbol  of such operators is integrable and we have 
\begin{equation}\label{TRTr}
\Int\sigma_A(\xi)=\int_{\mathbb{R}^n}\sigma_A(\xi)d\xi.
\end{equation}
 Therefore, the $\TR$-functional and operator trace coincide on classical pseudodifferential operators of order less than $-n$. 
 
The $\TR$-functional is in fact the analytic continuation of the operator trace and using this fact  we can prove that it is actually a trace.

 \begin{defi}
A family of symbols $\sigma(z)\in\mathcal{S}_{cl}^{\alpha(z)}(C^{\infty}(\mathbb{T}^n_{\theta}))$, parametrized by $z\in W\subset \mathbb{C}$, is called a holomorphic family if
\begin{itemize}
\item[i)]  The map $z\mapsto \alpha(z)$ is holomorphic. 
\item[ii)]  The map $z\mapsto \sigma(z)\in \mathcal{S}_{cl}^{\alpha(z)}(\A)$ is a holomorphic map from $W$ to the Fr\'{e}chet space $\mathcal{S}_{cl}(C^{\infty}(\mathbb{T}^n_{\theta})).$  
\item[iii)] The map $z\mapsto \sigma(z)_{\alpha(z)-j}$ is   holomorphic for any $j$, where 
 \begin{equation}\label{sigma2}
 \sigma(z)(\xi)\sim\sum_j\chi(\xi)\sigma(z)_{\alpha(z)-j}(\xi)\in\mathcal{S}_{cl}^{\alpha(z)}(C^{\infty}(\mathbb{T}^n_{\theta})).
 \end{equation}
\item[iv)] The bounds of the asymptotic expansion of  $\sigma(z)$ are locally uniform with respect to $z$, i.e, for any $L\geq1$ and compact subset $K\subset W$, there exists a constant $C_{L,K,\alpha,\beta}$ such that for all multi-indices  $\alpha,\beta$ we have
$$\left|\left|\delta^{\alpha}\partial^{\beta}\left(\sigma(z)-\sum_{j<L}\chi\sigma(z)_{\alpha(z)-j}\right)(\xi)\right|\right|<C_{L,K,\alpha,\beta}|\xi|^{\Re(\alpha(z))-L-|\beta|}.$$
\end{itemize}

   A family $\left\{A_z\right\}\in\Psi_{cl}(C^{\infty}(\mathbb{T}^n_{\theta}))$ is called holomorphic if $A_z=A_{\sigma(z)}$ for a holomorphic family of symbols $\left\{\sigma(z)\right\}$. 
  \end{defi}
  
 Complex powers of elliptic operators are an important class of holomorphic families. Let $Q\in\Psi^q_{cl}(C^{\infty}(\mathbb{T}^n_{\theta}))$ be a positive elliptic pseudodifferential operator of order $q>0$. 
 The complex power of such an operator, $Q_\phi^z$, for $\Re(z)<0$  can be defined by the following Cauchy integral formula.
\begin{equation}\label{Qz}
Q_\phi^z=\frac{i}{2\pi}\int_{C_\phi}\lambda_\phi^z (Q-\lambda)^{-1} d\lambda.
\end{equation}
Here $\lambda^z_\phi$ is the complex power with branch cut $L_\phi =\{re^{i\phi}, r\geq 0\}$ and $C_\phi$ is a contour around the spectrum of $Q$ such that
$$C_\phi\cap {\rm spec}(Q)\backslash\{0\}=\emptyset,\qquad L_\phi\cap C_\phi=\emptyset,$$
$$ C_\phi\cap\{ {\rm spec}(\sigma(Q)^L(\xi)),\,\xi\neq 0\}=\emptyset.$$
\begin{rmk}
\normalfont
More generally, an operator for which one can find a ray $L_\phi$ with the above property, is called an admissible operator with the spectral cut $L_\phi$ and its complex power can be defined as above. Self-adjoint elliptic operators are admissible (see \cite{Paycha2012, det-work}).
\end{rmk}

 The following Proposition is the analogue of the result of Kontsevich and Vishik \cite{Kontsevich-Vishik1995}, for pseudodifferential calculus on noncommutative tori.
 \begin{prop}\label{thmLaurent}
 Given a holomorphic family $\sigma(z)\in\mathcal{S}_{cl}^{\alpha(z)}(C^{\infty}(\mathbb{T}^n_{\theta}))$, $z\in W \subset\mathbb{C}$, the map 
 $$z\mapsto \Int\sigma(z)(\xi)d\xi,$$
 is meromorphic with at most simple poles located in 
 $$P=\left\{z_0\in W;~\alpha(z_0)\in\mathbb{Z}\cap[-n,+\infty]\right\}.$$
 The residues at poles are given by
 $$\Res_{z=z_0} \Int\sigma(z)(\xi)d\xi=-\frac{1}{\alpha'(z_0)}\int_{|\xi|=1}\sigma(z_0)_{-n}d\xi.$$
 \end{prop} 
 \begin{proof}
 By definition, one can write $\sigma(z)=\sum_{j=0}^L\chi(\xi)\sigma(z)_{\alpha(z)-j}(\xi)+\sigma(z)^L(\xi)$, and by Proposition (\ref{cut-off}) we have,
\begin{align*}
\Int\sigma(z)(\xi)d\xi
&=\int_{\mathbb{R}^n}\sigma(z)^L(\xi)d\xi
+\sum_{j=0}^{L}\int_{B(1)}\chi(\xi)\sigma(z)_{\alpha(z)-j}(\xi)\\
&-\sum_{j=0}^{L}\frac{1}{\alpha(z)+n-j}\int_{|\xi|=1}\sigma(z)_{\alpha(z)-j}(\xi)d\xi.
\end{align*} 
 Now suppose $\alpha(z_0)+n-j_0=0$. By holomorphicity of $\sigma(z)$, we have $\alpha(z)-\alpha(z_0)=\alpha'(z_0)(z-z_0)+o(z-z_0)$. Hence
 $$\Res_{z=z_0} \Int\sigma(z)=\frac{-1}{\alpha'(z_0)}\int_{|\xi|=1}\sigma(z_0)_{-n}(\xi)d\xi.$$
 \end{proof}
 \begin{cor}\label{analyticcontin}
 The functional $\TR$ is the analytic continuation of the ordinary trace on trace-class pseudodifferential operators.
 \end{cor}
 \begin{proof}
 First observe that, by the above result, for a non-integer order holomorphic family of symbols $\sigma(z)$, the map $z\mapsto\Int\sigma(z)(\xi)d\xi$ is holomorphic.  Hence, the map $\sigma\mapsto\Int\sigma(\xi)d\xi$ is the unique analytic continuation of the map $\sigma\mapsto\int_{\mathbb{R}^n}\sigma(\xi)d\xi$ from $\mathcal{S}_{cl}^{<-n}(C^{\infty}(\mathbb{T}^n_{\theta}))$ to $\mathcal{S}^{\notin\mathbb{Z}}_{cl}(C^{\infty}(\mathbb{T}^n_{\theta}))$.  By \eqref{TRTr} we have the result.
\end{proof}

\begin{cor}
 Let $A\in\Psi ^{\alpha}_{cl}(C^{\infty}(\mathbb{T}^n_{\theta}))$ be of order $\alpha\in\mathbb{Z}$ and let $Q$ be  a positive elliptic classical pseudodifferential operator of positive order $q$. We have
 $$\Res_{z=0}\TR(AQ^{-z})= \frac{1}{q}\Wres(A).$$
 \end{cor}
 \begin{proof}
 For the holomorphic family $\sigma(z)=\sigma(AQ^{-z})$,  $z=0$ is a pole for the map  $z\mapsto\Int\sigma(z)(\xi)d\xi$ whose  residue is given by 
 $$\Res_{z=0}\left(z\mapsto \Int\sigma(z)(\xi)d\xi\right)=-\frac{1}{\alpha'(0)}\int_{|\xi|=1}\sigma_{-n}(0)d\xi=-\frac{1}{\alpha'(0)}\res(A).$$
Taking $\tau$-trace on both sides gives the result. 
 \end{proof}
 Now we can prove the trace property of $\TR$-functional.
 \begin{prop}
 We have  
 $\TR(AB)=\TR(BA)$ for any $A,B\in\Psi^*_{cl}(C^{\infty}(\mathbb{T}^n_{\theta}))$, provided that $ord(A)+ord(B)\notin\mathbb{Z}$.
 \end{prop}
\begin{proof}
Consider the families $\{A_z\}$ and $\{B_z\}$ such that $A_0\sim A$, $B_0\sim B$ , $\ord(A_z)=\ord(A)+z$ and $\ord(B_z)=\ord(B)+z$.\
 For $z\in W=-(\ord(A)+\ord(B))+\mathbb{Z}$ the families $\{A_zB_z\}$ and $\{B_zA_z\}$ have non-integer order . For $\Re(z)\ll 0$, the two familes are trace-class and $\Tr(A_zB_z)=\Tr(B_zA_z)$. now by analytic continuation we have
 $TR(A_zB_z)=TR(B_zA_z)$, for $z\in\mathbb{C}-W$. Putting $z=0$ gives $TR(AB)=TR(BA)$.
 \end{proof}
 
 \begin{rmk}
 \normalfont
 The above result provides another proof for the trace property of the non-commutative residue on $\Psi^{\mathbb{Z}}_{cl}(C^{\infty}(\mathbb{T}^2_{\theta}))$ given in \cite{Farzad-Wong2011}, namely, for $A,B\in\Psi^{\mathbb{Z}}_{cl}(C^{\infty}(\mathbb{T}^2_{\theta}))$, $$\Wres([A,B])=0.$$
On can write, $$\Wres([A,B])=\Res_{z=0}\TR([A,B]Q^{-z})=\Res_{z=0}\TR(C_z)+\Res_{z=0}\TR([AQ^{-z},B]),$$
 where $C_z=ABQ^{-z}-AQ^{-z}B$. For $Re(z)\gg0$, the operator $AQ^{-z}$ is trace-class and  $\Tr([AQ^{-z},B])=0$, so by analytic continuation,$\TR([AQ^{-z},B])=0$ and
 therefore, $\Res_{z=0}\TR([A,B]Q^{-z})=\Res_{z=0}\TR(C_z)$. Finally, $C_0=ABQ^0-AQ^0B\in\Psi_{cl}^{-\infty}(C^{\infty}(\mathbb{T}^2_{\theta}))$, so
 $$\Wres([A,B])=\Res_{z=0}\TR(C_z)=\Wres(C_0)=0,$$
 where in the last equality we used the fact that the noncommutative residue of a smoothing operator is zero.
 \end{rmk}
 \subsection{Log-polyhomogeneous symbols}  
In general, $z$-derivatives of a classical holomorphic family of symbols are not classical anymore and therefore we introduce log-polyhomogeneous symbols which include the $z$-derivatives of the symbols of the holomorphic family $\sigma(AQ^{-z})$. 

\begin{defi}
A symbol $\sigma$ is called a log-polyhomogeneous symbol if it has the following form
\begin{equation}\label{logpolysym}
\sigma(\xi) \sim \sum_{j\geq 0}\sum_{l=0}^\infty\sigma_{\alpha-j,l}(\xi)\log^l|\xi|\quad |\xi|>0,
\end{equation}
with $ \sigma_{\alpha-j,l}$ positively homogeneous in $\xi$ of degree $\alpha - j$.
\end{defi}

 An important example  of an operator with such a symbol is $\log Q$ where $Q\in\Psi^q_{cl}(C^{\infty}(\mathbb{T}^n_{\theta}))$ is a positive elliptic pseudodifferential operator of order $q>0$.
 The logarithm of $Q$ can be  defined by
 $$\log Q
 =Q\left.\frac{d}{dz}\right|_{z=0}Q^{z-1}
 =Q\left.\frac{d}{dz}\right|_{z=0}\frac{i}{2\pi}\int_C\lambda^{z-1}(Q-\lambda)^{-1}d\lambda.$$ 
 It is a pseudodifferential operator with symbol 
\begin{equation}\label{symboflog}\sigma(\log Q)
\sim \sigma(Q)\star \sigma\Big(\left.\frac{d}{dz}\right|_{z=0} Q^{z-1}\Big),
\end{equation}
where $\star$ denotes the product of symbols.  
One can show that \eqref{symboflog} is a log-polyhomogeneous symbol of the form 
 $$\sigma(\log Q)(\xi)=q\log|\xi| I+ \sigma_{cl}(\log Q)(\xi),$$ 
 where  $\sigma_{cl}(\log Q)$ is a classical symbol of order zero (see \cite{Paycha2012}).

By adapting the proof of Theorem 1.13 in \cite{Paycha-Scott2007} to the  noncommutative case, we have the following theorem for the family $\sigma(AQ^{-z})$. 
 
 \begin{prop}\label{Laurent exp}
Let $A\in\Psi_{cl}^\alpha(C^{\infty}(\mathbb{T}^n_{\theta}))$ and $Q$ be a positive (or more generally, an admissible) elliptic pseudodifferential operator of positive order $q$. If $\alpha\in P$ then $z=0$ is a possible simple pole for the function $z\mapsto \TR(AQ^{-z})$ with the following Laurent expansion around zero,
\begin{align*}
\TR(AQ^{-z})&=\frac{1}{q}\Wres(A)\frac{1}{z}\\
 &+\tau\left(\Int\sigma(A)- \frac{1}{q}\res(A\log Q)\right)-\Tr(A\Pi_Q)\\
& +\sum_{k=1}^K(-1)^k\frac{(z)^k}{k!} \\
&\times \left(\tau\left( \Int\sigma(A(\log Q)^k)d\xi-\frac{1}{q(k+1)}\res(A(\log Q)^{k+1})\right)-\Tr(A\log^k Q\Pi_Q)\right)\\
&+o(z^{K}).
\end{align*}
Where $\Pi_Q$ is the projection on the kernel of $Q$.
\end{prop} \qed

\begin{rmk}\normalfont
The term $\res(A\log Q)$ appearing in above Laurent expansion is an extension of Wodzicki residue density to operators with Log-polyhomogeneous symbols \cite{Lesch1999}.
For an operator  $P$ with log-polyhomogeneous symbol, by $\res(P)$ we mean,
 $$\res(P)=\int_{|\xi|=1}\sigma_{-n,0}(\xi)d\xi,$$
  (see \eqref{logpolysym}).
\end{rmk}

 \section{The spectral eta function}
 In this section we study the eta function associated with the family of spectral triples $(\T^3,\mathcal{H},e^{th}De^{th})$ where $h\in\T^3$ is a self-adjoint element \cite{Connes-Moscovici-twisted}, and also the coupled spectral triple $(\T^3,\mathcal{H},D+A)$. By exploiting the developed pseudodifferential calculus, the regularity of the eta function at zero in above cases will be proved. 
 \subsection{Regularity at zero}
  The spectral eta function was first introduced in \cite{APS} where its value at zero appeared as a correction term in the Atiyah-Patodi-Singer index theorem for manifolds with boundary. It is defined as 
 $$\eta_D(z)=\sum_{\lambda\in\spec(D),\lambda\neq0}\sgn(\lambda)|\lambda|^{-z}=\TR\left(D|D|^{-z-1}\right),$$
 where $D$ is a self-adjoint elliptic pseudodifferential operator.
 Unlike the spectral zeta functions for positive elliptic operators, proving the regularity of eta function at zero is difficult. This was proved in \cite{APSIII} and \cite{Gilkey-eta} using K-theoretic arguments and in \cite{Bismut-Freed}, Bismut and Freed gave an analytic proof of the regularity at zero of the eta function for a twisted Dirac operator on an odd dimensional spin manifold.
 
\begin{rmk}
\normalfont
Note that for an even spectral triple $(\mathcal{A},H,D)$ we have $D\Gamma=-\Gamma D$, therefore the spectrum of the Dirac operator 
is symmetric and $\eta_D(z)$ is identically zero. Also the same vanishing happens if the spectral triple admits a real structure with KO-dimensions $1$ or $5$ (see definition \ref{real-table}).
\end{rmk} 
 The meromorphic structure of eta function for Dirac operator can be studied by the pole structure of the canonical trace for holomorphic families.
  \begin{prop}\label{eta-res}
 Let $D$ be an elliptic self-adjoint first-order differential operator on $\T^3$. The poles of the eta function $\eta_D(z)$ are located among
  $\{3-i,~i\in\mathbb{N}\}$, and 
$$\Res_{z=0}\eta_D(z)=\Wres(D|D|^{-1}).$$
 \end{prop}
 \begin{proof}
 By  using the result of Proposition \ref{thmLaurent},
the family $\left\{\sigma\left(D|D|^{-z-1}\right)\right\}$ has poles within the set $\left\{z;-z\in\mathbb{Z}\cap[-3,\infty]\right\}$ or  $\left\{z=3-i,i\in\mathbb{N}\right\}$.
Also, by Proposition \ref{Laurent exp} we have
\begin{equation*}
\eta_D(z)=\TR\left(D|D|^{-z-1}\right)=\Wres(D|D|^{-1})\frac{1}{z}+a_0+a_1z+\cdots,
\end{equation*}
Hence,
$$\Res_{z=0}\eta_D(z)=\Wres(D|D|^{-1}).$$
 \end{proof}
 We now prove the regularity at $z=0$ of eta function for the 1-parameter family $\{e^{th}De^{th}\}$ for the spectral triple $(\T^3,H_{\tau}\otimes\mathbb{C}^2,D=\partial_{\mu}\otimes\gamma^{\mu})$ on noncommutative 3-torus.

 \begin{prop}\label{eta-reg}
 Consider the family of operators $\{e^{th}De^{th}\}$ on $\mathcal{H}_{\tau}\otimes\mathbb{C}^2$ where $h=h^*\in\T^3$, then $$\Res_{z=0} \eta_{e^{th}De^{th}}(z)=0.$$  
  \end{prop}
  \begin{proof}
  By definition, $\eta_{D_t}(z)=\TR(D_t|D_t|^{-z-1})$. Using Proposition \ref{eta-res} we have
  $$\Res_{z=0} \eta_{D_t}(z)=\Wres(D_t|D_t|^{-1}),$$
  where the right hand side is the Wodzicki residue on noncommutative $3$-torus. 
   
 Now for each element of the family $D_t=e^{\frac{th}{2}}De^{\frac{th}{2}}$, $D_t^2=e^{\frac{th}{2}}De^{th}De^{\frac{th}{2}}$ and $|D_t|=\sqrt{D_t^2}$. By using the product formula for the symbols we have,
 \begin{equation}
 \sigma(D_t)\sim\xi_{\mu}e^{th}\otimes\gamma^{\mu}+e^{\frac{th}{2}}\delta_{\mu}(e^{\frac{th}{2}})\otimes\gamma^{\mu},
 \end{equation}
and
\begin{align*}
\sigma(D_t^2)&=\sigma(e^{\frac{th}{2}}De^{th}De^{\frac{th}{2}})\sim\xi_{\lambda}\xi_{\mu}e^{2th}\otimes\gamma^{\lambda}\gamma^{\mu}\\
&+\xi_{\lambda}e^{\frac{3th}{2}}\delta_{\mu}(e^{\frac{th}{2}})\gamma^{\lambda}\gamma^{\mu}+\xi_{\mu}e^{\frac{th}{2}}\delta_{\lambda}(e^{\frac{3th}{2}})\otimes\gamma^{\lambda}\gamma^{\mu}\\
&+e^{\frac{th}{2}}\delta_{\lambda}(e^{th})\delta_{\mu}(e^{\frac{th}{2}})\otimes\gamma^{\lambda}\gamma^{\mu}+e^{\frac{3th}{2}}\delta_{\lambda}\delta_{\mu}(e^{\frac{th}{2}})\otimes\gamma^{\lambda}\gamma^{\mu}.\\
\end{align*}
To compute the homogeneous terms in the symbol of $|D_t|$, we observe that $|D_t|=\sqrt{D_t^2}$ and hence,
\begin{equation}
\sigma(|D_t|)=\sqrt{\sigma(D_t^2)}\sim\sigma_1(\xi)+\sigma_0(\xi)+\sigma_{-1}(\xi)+\cdots.
\end{equation}
We compute the first three terms, which we need for computing the symbol of $|D_t|^{-1}$.
\begin{align*}
\sigma_1(\xi)&=\lim_{k\to\infty}\frac{\sigma(|D_t|)(k\xi)}{k}=\sqrt{\left(\xi^2\right)}e^{th}\otimes I,\\
\sigma_0(\xi)&=\lim_{k\to\infty}\sigma(|D_t|)(k\xi)-k\sigma_1(\xi)\\
&=\left(\xi_{\lambda}e^{\frac{3th}{2}}\delta_{\mu}(e^{\frac{th}{2}})\right)(\frac{1}{2\sqrt{\xi^2}}e^{-th})\otimes\gamma^{\lambda}\gamma^{\mu}+\left(\xi_{\mu}e^{\frac{th}{2}}\delta_{\lambda}(e^{\frac{3th}{2}})\right)(\frac{1}{2\sqrt{\xi^2}}e^{-th})\otimes\gamma^{\lambda}\gamma^{\mu},\\
\sigma_{-1}(\xi)&=\lim_{k\to\infty}\frac{\left(\sigma(|D_t|)(k\xi)-\sigma_1(\xi)-\sigma_0(\xi)\right)}{k^{-1}}\\
&=\frac{1}{2\sqrt{\xi^2}}\left(e^{\frac{th}{2}}\delta_{\lambda}(e^{th})\delta_{\mu}(e^{\frac{th}{2}})\right)e^{-th}\otimes\gamma^{\lambda}\gamma^{\mu}+\frac{1}{2\sqrt{\xi^2}}\left(e^{\frac{3th}{2}}\delta_{\lambda}\delta_{\mu}(e^{\frac{th}{2}})\right)e^{-th}\otimes\gamma^{\lambda}\gamma^{\mu}\\
&-\frac{1}{8\xi^2\sqrt{\xi^2}}\left(\xi_{\lambda}e^{\frac{3th}{2}}\delta_{\mu}(e^{\frac{th}{2}})\right)e^{-th}\left(\xi_{\nu}e^{\frac{3th}{2}}\delta_{\rho}(e^{\frac{th}{2}})\right)e^{-2th}\otimes\gamma^{\lambda}\gamma^{\mu}\gamma^{\nu}\gamma^{\rho}\\
&-\frac{1}{8\xi^2\sqrt{\xi^2}}\left(\xi_{\lambda}e^{\frac{3th}{2}}\delta_{\mu}(e^{\frac{th}{2}})\right)e^{-th}\left(\xi_{\rho}e^{\frac{th}{2}}\delta_{\nu}(e^{\frac{3th}{2}})\right)e^{-2th}\otimes\gamma^{\lambda}\gamma^{\mu}\gamma^{\nu}\gamma^{\rho}\\
&-\frac{1}{8\xi^2\sqrt{\xi^2}}\left(\xi_{\mu}e^{\frac{th}{2}}\delta_{\lambda}(e^{\frac{3th}{2}})\right)e^{-th}\left(\xi_{\nu}e^{\frac{3th}{2}}\delta_{\rho}(e^{\frac{th}{2}})\right)e^{-2th}\otimes\gamma^{\lambda}\gamma^{\mu}\gamma^{\nu}\gamma^{\rho}\\
&-\frac{1}{8\xi^2\sqrt{\xi^2}}\left(\xi_{\mu}e^{\frac{th}{2}}\delta_{\lambda}(e^{\frac{3th}{2}})\right)e^{-th}\left(\xi_{\rho}e^{\frac{th}{2}}\delta_{\nu}(e^{\frac{3th}{2}})\right)e^{-2th}\otimes\gamma^{\lambda}\gamma^{\mu}\gamma^{\nu}\gamma^{\rho},\\
\end{align*}
where we have used the notation $\xi^2=\sum_{k=0}^3\xi_k^2$.

Now by using the relation $|D_t|^{-1}|D_t|=1$ and by recursive computation we find the homogeneous terms in the symbol of the inverse,
\begin{equation}
\sigma(|D_t|^{-1})\sim\sigma_{-1}(|D_t|^{-1})(\xi)+\sigma_{-2}(|D_t|^{-1})(\xi)+\cdots,
\end{equation}
where,
$$\sigma_{-1}(|D_t|^{-1})(\xi)=\frac{1}{\sqrt{\xi^2}}e^{-th}\otimes I,$$
\begin{align*}
\sigma_{-2}(|D_t|^{-1})(\xi)&=\\
&-\sigma_{-1}(|D_t|^{-1})\{ \sigma_0(|D_t|)\sigma_{-1}(|D_t|^{-1})\\
&+\sum_{\alpha_1+\alpha_2+\alpha_3=1}\partial_{\xi_1}^{\alpha_1}\partial_{\xi_2}^{\alpha_2}\partial_{\xi_3}^{\alpha_3}\sigma_{1}(|D_t|)\delta^{\alpha_1}\delta^{\alpha_2}\delta^{\alpha_3}\sigma_{-1}(|D_t|^{-1})\},\\
\sigma_{-3}\{|D_t|^{-1})(\xi)&=\\
&-\sigma_{-1}(|D_t|^{-1})\{ (\sigma_{-1}(|D_t|)\sigma_{-1}(|D_t|^{-1})+\sigma_0(|D_t|)\sigma_{-2}(|D_t|^{-1})\\
&+\sum_{\alpha_1+\alpha_2+\alpha_3=2}\frac{1}{\alpha_1!\alpha_2!\alpha_3!}\partial_{\xi_1}^{\alpha_1}\partial_{\xi_2}^{\alpha_2}\partial_{\xi_3}^{\alpha_3}\sigma_{1}(|D_t|)\delta^{\alpha_1}\delta^{\alpha_2}\delta^{\alpha_3}\sigma_{-1}(|D_t|^{-1})\\
&+\sum_{\alpha_1+\alpha_2+\alpha_3=1}\partial_{\xi_1}^{\alpha_1}\partial_{\xi_2}^{\alpha_2}\partial_{\xi_3}^{\alpha_3}\sigma_{1}(|D_t|)\delta^{\alpha_1}\delta^{\alpha_2}\delta^{\alpha_3}\sigma_{-2}(|D_t|^{-1})\},
\end{align*}
\begin{align*}
\sigma_{-4}(|D_t|^{-1})(\xi)&=\\
-\sigma_{-1}(|D_t|^{-1})&\{ \sigma_{-2}(|D_t|)\sigma_{-1}(|D_t|^{-1})
+\sigma_{-1}(|D_t|)\sigma_{-2}(|D_t|^{-1})
+\sigma_0(|D_t|)\sigma_{-3}(|D_t|^{-1})\\
&+\sum_{\alpha_1+\alpha_2+\alpha_3=1}\frac{1}{\alpha_1!\alpha_2!\alpha_3!}\partial_{\xi_1}^{\alpha_1}\partial_{\xi_2}^{\alpha_2}\partial_{\xi_3}^{\alpha_3}\sigma_0(|D_t|)\delta^{\alpha_1}\delta^{\alpha_2}\delta^{\alpha_3}\sigma_{-2}(|D_t|^{-1})\\
&+\sum_{\alpha_1+\alpha_2+\alpha_3=1}\frac{1}{\alpha_1!\alpha_2!\alpha_3!}\partial_{\xi_1}^{\alpha_1}\partial_{\xi_2}^{\alpha_2}\partial_{\xi_3}^{\alpha_3}\sigma_{-1}(|D_t|)\delta^{\alpha_1}\delta^{\alpha_2}\delta^{\alpha_3}\sigma_{-1}(|D_t|^{-1})\\
&+\sum_{\alpha_1+\alpha_2+\alpha_3=1}\frac{1}{\alpha_1!\alpha_2!\alpha_3!}\partial_{\xi_1}^{\alpha_1}\partial_{\xi_2}^{\alpha_2}\partial_{\xi_3}^{\alpha_3}\sigma_1(|D_t|)\delta^{\alpha_1}\delta^{\alpha_2}\delta^{\alpha_3}\sigma_{-3}(|D_t|^{-1})\\
&+\sum_{\alpha_1+\alpha_2+\alpha_3=2}\frac{1}{\alpha_1!\alpha_2!\alpha_3!}\partial_{\xi_1}^{\alpha_1}\partial_{\xi_2}^{\alpha_2}\partial_{\xi_3}^{\alpha_3}\sigma_0(|D_t|)\delta^{\alpha_1}\delta^{\alpha_2}\delta^{\alpha_3}\sigma_{-1}(|D_t|^{-1})\\
&+\sum_{\alpha_1+\alpha_2+\alpha_3=2}\frac{1}{\alpha_1!\alpha_2!\alpha_3!}\partial_{\xi_1}^{\alpha_1}\partial_{\xi_2}^{\alpha_2}\partial_{\xi_3}^{\alpha_3}\sigma_1(|D_t|)\delta^{\alpha_1}\delta^{\alpha_2}\delta^{\alpha_3}\sigma_{-2}(|D_t|^{-1})\\
&+\sum_{\alpha_1+\alpha_2+\alpha_3=3}\frac{1}{\alpha_1!\alpha_2!\alpha_3!}\partial_{\xi_1}^{\alpha_1}\partial_{\xi_2}^{\alpha_2}\partial_{\xi_3}^{\alpha_3}\sigma_1(|D_t|)\delta^{\alpha_1}\delta^{\alpha_2}\delta^{\alpha_3}\sigma_{-1}(|D_t|^{-1})\}.
\end{align*}
 Therefore, the symbol of $\sigma(D_t|D_t|^{-1})$ reads
   \begin{align*}  
       &\sigma(D_t|D_t|^{-1})\sim\left(\sigma_1(D_t)+\sigma_0(D_t)\right)\star\left(\sigma_{-1}(|D_t|^{-1})+\sigma_{-2}(|D_t|^{-1})+\sigma_{-3}(|D_t|^{-1})+\cdots\right)\\
  &\sim\left(\sigma_{1}(D_t)\star\sigma_{-1}(|D_t|^{-1})\right)+\left(\sigma_{1}(D_t)\star\sigma_{-2}(|D_t|^{-1})\right)\\
  &+\left(\sigma_{1}(D_t)\star\sigma_{-3}(|D_t|^{-1})\right)+\left(\sigma_1(D_t)\star\sigma_{-4}(|D_t|^{-1})\right)+\cdots\\
  &+\left(\sigma_{0}(D_t)\star\sigma_{-1}(|D_t|^{-1})\right)+\left(\sigma_{0}(D_t)\star\sigma_{-2}(|D_t|^{-1})\right)+\left(\sigma_0(D_t)\star\sigma_{-3}(|D_t|^{-1})\right)+\cdots.\\ 
   \end{align*}
   For $a=1,0$ and $b=-1,-2,-3.\cdots$, one has
   $$\sigma_a(D_t)\star\sigma_b(|D_t|^{-1})=\sum_{\alpha}\frac{1}{\alpha!}\partial_{\xi}^{\alpha}\sigma_a(D_t)\delta^{\alpha}\sigma_b(|D_t|^{-1}),$$
   and each term is of order $a-|\alpha|+b$. By collecting the terms of order $-3$ we obtain,
     \begin{align*}
   \sigma_{-3}(D_t|D_t|^{-1})&\sim\left(\sigma_1(D_t)\sigma_{-4}(|D_t|^{-1})\right)+\left(\sigma_0(D_t)\sigma_{-3}(|D_t|^{-1})\right)\\
   &+\left(\sum_{\alpha_1+\alpha_2+\alpha_3=1}\partial_{\xi_1}^{\alpha_1}\partial_{\xi_2}^{\alpha_2}\partial_{\xi_3}^{\alpha_3}\sigma_{1}(D_t)\delta_1^{\alpha_1}\delta_2^{\alpha_2}\delta_3^{\alpha_3}\sigma_{-3}(|D_t|^{-1})\right)\\
   &+\left(\sum_{\alpha_1+\alpha_2+\alpha_3=2}\frac{1}{\alpha_1!\alpha_2!\alpha_3!}\partial_{\xi_1}^{\alpha_1}\partial_{\xi_2}^{\alpha_2}\partial_{\xi_3}^{\alpha_3}\sigma_{1}(D_t)\delta_1^{\alpha_1}\delta_2^{\alpha_2}\delta_3^{\alpha_3}\sigma_{-2}(|D_t|^{-1})\right)\\
   &+\left(\sum_{\alpha_1+\alpha_2+\alpha_3=1}\partial_{\xi_1}^{\alpha_1}\partial_{\xi_2}^{\alpha_2}\partial_{\xi_3}^{\alpha_3}\sigma_{0}(D_t)\delta_1^{\alpha_1}\delta_2^{\alpha_2}\delta_3^{\alpha_3}\sigma_{-2}(|D_t|^{-1})\right)\\
   &+\left(\sum_{\alpha_1+\alpha_2+\alpha_3=2}\frac{1}{\alpha_1!\alpha_2!\alpha_3!}\partial_{\xi_1}^{\alpha_1}\partial_{\xi_2}^{\alpha_2}\partial_{\xi_3}^{\alpha_3}\sigma_{0}(D_t)\delta_1^{\alpha_1}\delta_2^{\alpha_2}\delta_3^{\alpha_3}\sigma_{-1}(|D_t|^{-1})\right).\\
   \end{align*}
   Now, one should notice that Wodzicki residue of a matrix pseudodifferential operator involves a trace taken over  the matrix coefficients. By analyzing the terms involved in $\sigma_{-3}(D_t|D_t|^{-1})$ and using the trace identities for gamma matrices (cf. e.g. \cite{Folland-book}),
 we see that the only contribution is from the the terms whose matrix coefficient consist of either one $\gamma$ matrix or product of three. Next, we use the following identities,
  \begin{align}
&\tr(\gamma^{\lambda})=0,\\
&\tr(\gamma^{\lambda}\gamma^{\mu}\gamma^{\nu})=i\epsilon^{\lambda\mu\nu}\tr(I),
\end{align}
where $\epsilon^{\lambda\mu\nu}$ is the Levi-Civita symbol.
By observing that
$$\sum_{\lambda,\mu,\nu}\epsilon^{\lambda,\mu,\nu}=0,$$
 we conclude that
\begin{equation}
\Wres\left(D_t|D_t|^{-1}\right)=\tau\left(\int_{|\xi|=1}\tr\left(\res(D_t|D_t|^{-1}\right)d\xi\right)=0.
\end{equation}
 \end{proof}
 \begin{rmk}\normalfont
 In \cite{Bismut-Freed}, Bismut-Freed showed that in fact, for a twisted Dirac operator on a Riemannian Spin manifold, the {\it local eta-residue}, $\tr \left(\res_x\left(D|D|^{-1}\right)\right)$ vanishes. The above and the following results confirm the same vanishing in the noncommutative case $\T^3$.
 \end{rmk}
    The following result proves the regularity at zero of eta function for the coupled Dirac operator $D+A$ on $\T^3$.
    \begin{prop}\label{eta-coupled}
    Consider the coupled Dirac operator $\slashed{\partial}+\slashed{A}$ on $\mathcal{H}_{\tau}\otimes\mathbb{C}^2$ over $\T^3$, then $$\Res_{z=0} \eta_{D+A}(z)=0.$$  
    \end{prop}
    \begin{proof}
 \
The coupled Dirac operator $D$ and $D^2$ are given by
\begin{align*}
D&=\partial_{\mu}\otimes\gamma^{\mu}+A_{\mu}\otimes\gamma^{\mu}\\
D^2&=\partial_{\mu}\partial_{\lambda}\otimes\gamma^{\mu}\gamma^{\lambda}+\partial_{\mu}(A_{\lambda})\otimes\gamma^{\mu}\gamma^{\lambda}+
   A_{\mu}\partial_{\lambda}\otimes\gamma^{\mu}\gamma^{\lambda}+A_{\mu}A_{\lambda}\otimes\gamma^{\mu}\gamma^{\lambda}.\\
   \end{align*}
   Note that $\partial_{\mu}A_{\lambda}(a)=\partial_{\mu}(A_{\lambda})a+A_{\lambda}\partial_{\mu}(a)$ and also $A_{\mu}A_{\lambda}\neq A_{\lambda}A_{\mu}$.
   Hence,
      $$D^2=\sum_{\mu}\partial_{\mu}^2\otimes I+2A_{\mu}\partial_{\lambda}\otimes\gamma^{\mu}\gamma^{\lambda}+\partial_{\mu} (A_{\lambda})\otimes\gamma^{\mu}\gamma^{\lambda}+A_{\mu}A_{\lambda}\otimes\gamma^{\mu}\gamma^{\lambda},$$
where $\partial_{\mu} (A_{\lambda})$ in the third term above is a multiplication operator.
The symbols are:
   $$\sigma(D)=\slashed{\xi}+\slashed{A},$$
  and
  $$\sigma(D^2)=\xi^2+2A_{\mu}\xi_{\lambda}\otimes\gamma^{\mu}\gamma^{\lambda}+\partial_{\mu} (A_{\lambda})\otimes\gamma^{\mu}\gamma^{\lambda}+A_{\mu}A_{\lambda}\otimes\gamma^{\mu}\gamma^{\lambda},$$
  where $\xi^2=\sum_{\mu}\xi_{\mu}^2\otimes I$.
   Now,
   \begin{align*}
   \sigma(|D|)&=\sqrt{\xi^2+2A_{\mu}\xi_{\lambda}\otimes\gamma^{\mu}\gamma^{\lambda}+\partial_{\mu} (A_{\lambda})\otimes\gamma^{\mu}\gamma^{\lambda}+A_{\mu}A_{\lambda}\otimes\gamma^{\mu}\gamma^{\lambda}}\\
   &\sim \sigma_1(\xi)+\sigma_0(\xi)+\sigma_{-1}(\xi)+\cdots,
   \end{align*}
 where
 \begin{align*}
\sigma_1(\xi)&=\sqrt{\xi^2}\otimes I,\\
\sigma_0(\xi)&=\frac{\slashed{A}\slashed{\xi}}{\sqrt{\xi^2}},\\
   \sigma_{-1}(\xi)&=\frac{1}{2\sqrt{\xi^2}}\left(\partial_{\mu}(A_{\lambda})\otimes\gamma^{\mu}\gamma^{\lambda}+A_{\mu}A_{\lambda}\otimes\gamma^{\mu}\gamma^{\lambda}\right)\\
&-\frac{1}{2\xi^2\sqrt{\xi^2}}\left(A_{\mu}\xi_{\lambda}A_{\nu}\xi_{\rho}\otimes\gamma^{\mu}\gamma^{\lambda}\gamma^{\nu}\gamma^{\rho}\right).
\end{align*}
  By using $\sigma(|D|)\star\sigma(|D|^{-1})\sim1$ and by recursive computation, we get:
   $$\sigma_{-1}(|D|^{-1})=\frac{1}{\sqrt{\xi^2}}\otimes I.$$
  \begin{align*}   
   \sigma_{-2}(|D|^{-1})&=\\
   -\frac{1}{\sqrt{\xi^2}}\otimes I&\left(\sigma_0(|D|)\sigma_{-1}(|D|^{-1})+\sum_{\alpha_1+\alpha_2+\alpha_3=1}\partial_{\xi_1}^{\alpha_1}\partial_{\xi_2}^{\alpha_2}\partial_{\xi_3}^{\alpha_3}\sigma_{1}(|D|)\delta_1^{\alpha_1}\delta_2^{\alpha_2}\delta_3^{\alpha_3}\sigma_{-1}(|D|^{-1})\right)\\
   &=-\frac{1}{\sqrt{\xi^2}}\otimes I\left(\frac{\slashed{A}\slashed{\xi}}{\sqrt{\xi^2}}.\frac{1}{\sqrt{\xi^2}}\otimes I\right)=-\frac{\slashed{A}\slashed{\xi}}{\xi^2\sqrt{\xi^2}}.\\
   \end{align*}
   \begin{align*}
   \sigma_{-3}(|D|^{-1})&=-\frac{1}{\sqrt{\xi^2}}\otimes I\{ \sigma_{-1}(|D|)\sigma_{-1}(|D|^{-1})+\sigma_0(|D|)\sigma_{-2}(|D|^{-1})\\
   &+\sum_{\alpha_1+\alpha_2+\alpha_3=1}\partial_{\xi_1}^{\alpha_1}\partial_{\xi_2}^{\alpha_2}\partial_{\xi_3}^{\alpha_3}\sigma_{0}(|D|)\delta_1^{\alpha_1}\delta_2^{\alpha_2}\delta_3^{\alpha_3}\sigma_{-1}(|D|^{-1})\\
   &+\sum_{\alpha_1+\alpha_2+\alpha_3=1}\partial_{\xi_1}^{\alpha_1}\partial_{\xi_2}^{\alpha_2}\partial_{\xi_3}^{\alpha_3}\sigma_{1}(|D|)\delta_1^{\alpha_1}\delta_2^{\alpha_2}\delta_3^{\alpha_3}\sigma_{-2}(|D|^{-1})\\
  &+\sum_{\alpha_1+\alpha_2+\alpha_3=2}\frac{1}{\alpha_1!\alpha_2!\alpha_3!}\partial_{\xi_1}^{\alpha_1}\partial_{\xi_2}^{\alpha_2}\partial_{\xi_3}^{\alpha_3}\sigma_{1}(|D|)\delta_1^{\alpha_1}\delta_2^{\alpha_2}\delta_3^{\alpha_3}\sigma_{-1}(|D|^{-1})\}\\
   &=-\frac{1}{\sqrt{\xi^2}}\otimes I\left(\frac{\slashed{A}\slashed{\xi}}{\sqrt{\xi^2}}.\frac{-\slashed{A}\slashed{\xi}}{\xi^2\sqrt{\xi^2}}+\partial_{\xi}^1(\sqrt{\xi^2}\otimes I).\delta^1(\frac{-\slashed{A}\slashed{\xi}}{\xi^2\sqrt{\xi^2}})\right)\\
   &=\frac{\xi_{\mu}\xi_{\nu}A_{\rho}A_{\lambda}}{\xi^4\sqrt{\xi^2}}\otimes\gamma^{\mu}\gamma^{\nu}\gamma^{\rho}\gamma^{\lambda}+\frac{1}{\sqrt{\xi^2}}\partial_{\xi}^1(\sqrt{\xi^2})\delta^1(\frac{A_{\mu}\xi_{\nu}}{\xi^2\sqrt{\xi^2}})\otimes\gamma^{\mu}\gamma^{\nu}.\\
   \end{align*}
   \begin{align*}
   \sigma_{-4}(|D|^{-1})&=-\frac{1}{\sqrt{\xi^2}}\otimes I\{\sigma_{-2}(|D|)\sigma_{-1}(|D|^{-1})+\sigma_{-1}(|D|)\sigma_{-2}(|D|^{-1})\\
   &+\sigma_0(|D|)\sigma_{-3}(|D|^{-1})\\
   &+\sum_{\alpha_1+\alpha_2+\alpha_3=1}\partial_{\xi_1}^{\alpha_1}\partial_{\xi_2}^{\alpha_2}\partial_{\xi_3}^{\alpha_3}\sigma_{-1}(|D|)\delta_1^{\alpha_1}\delta_2^{\alpha_2}\delta_3^{\alpha_3}\sigma_{-1}(|D|^{-1})\\
   &+\sum_{\alpha_1+\alpha_2+\alpha_3=1}\partial_{\xi_1}^{\alpha_1}\partial_{\xi_2}^{\alpha_2}\partial_{\xi_3}^{\alpha_3}\sigma_{0}(|D|)\delta_1^{\alpha_1}\delta_2^{\alpha_2}\delta_3^{\alpha_3}\sigma_{-2}(|D|^{-1})\\
   &+\sum_{\alpha_1+\alpha_2+\alpha_3=2}\frac{1}{\alpha_1!\alpha_2!\alpha_3!}\partial_{\xi_1}^{\alpha_1}\partial_{\xi_2}^{\alpha_2}\partial_{\xi_3}^{\alpha_3}\sigma_{0}(|D|)\delta_1^{\alpha_1}\delta_2^{\alpha_2}\delta_3^{\alpha_3}\sigma_{-1}(|D|^{-1})\\
  &+\sum_{\alpha_1+\alpha_2+\alpha_3=1}\partial_{\xi_1}^{\alpha_1}\partial_{\xi_2}^{\alpha_2}\partial_{\xi_3}^{\alpha_3}\sigma_{1}(|D|)\delta_1^{\alpha_1}\delta_2^{\alpha_2}\delta_3^{\alpha_3}\sigma_{-3}(|D|^{-1})\\
   &+\sum_{\alpha_1+\alpha_2+\alpha_3=2}\frac{1}{\alpha_1!\alpha_2!\alpha_3!}\partial_{\xi_1}^{\alpha_1}\partial_{\xi_2}^{\alpha_2}\partial_{\xi_3}^{\alpha_3}\sigma_{1}(|D|)\delta_1^{\alpha_1}\delta_2^{\alpha_2}\delta_3^{\alpha_3}\sigma_{-2}(|D|^{-1})\\
   &+\sum_{\alpha_1+\alpha_2+\alpha_3=3}\frac{1}{\alpha_1!\alpha_2!\alpha_3!}\partial_{\xi_1}^{\alpha_1}\partial_{\xi_2}^{\alpha_2}\partial_{\xi_3}^{\alpha_3}\sigma_{1}(|D|)\delta_1^{\alpha_1}\delta_2^{\alpha_2}\delta_3^{\alpha_3}\sigma_{-1}(|D|^{-1})\}\\
   &=-\frac{1}{\sqrt{\xi^2}}\otimes I\{\sigma_0(|D|)\sigma_{-3}(|D|^{-1})+\sum_{\alpha_1+\alpha_2+\alpha_3=1}\partial_{\xi_1}^{\alpha_1}\partial_{\xi_2}^{\alpha_2}\partial_{\xi_3}^{\alpha_3}\sigma_{0}(|D|)\delta_1^{\alpha_1}\delta_2^{\alpha_2}\delta_3^{\alpha_3}\sigma_{-2}(|D|^{-1})\\
   &+\sum_{\alpha_1+\alpha_2+\alpha_3=1}\partial_{\xi_1}^{\alpha_1}\partial_{\xi_2}^{\alpha_2}\partial_{\xi_3}^{\alpha_3}\sigma_{1}(|D|)\delta_1^{\alpha_1}\delta_2^{\alpha_2}\delta_3^{\alpha_3}\sigma_{-3}(|D|^{-1})\\
   &+\sum_{\alpha_1+\alpha_2+\alpha_3=2}\frac{1}{\alpha_1!\alpha_2!\alpha_3!}\partial_{\xi_1}^{\alpha_1}\partial_{\xi_2}^{\alpha_2}\partial_{\xi_3}^{\alpha_3}\sigma_{1}(|D|)\delta_1^{\alpha_1}\delta_2^{\alpha_2}\delta_3^{\alpha_3}\sigma_{-2}(|D|^{-1})\}.
   \end{align*}
   Therefore the symbol of $\sigma(D|D|^{-1})$ reads
   \begin{align*}  
       &\sigma(D|D|^{-1})\sim\left(\sigma_1(D)+\sigma_0(D)\right)\star\left(\sigma_{-1}(|D|^{-1})+\sigma_{-2}(|D|^{-1})+\sigma_{-3}(|D|^{-1})+\cdots\right)\\
  &\sim\left(\sigma_{1}(D)\star\sigma_{-1}(|D|^{-1})\right)+\left(\sigma_{1}(D)\star\sigma_{-2}(|D|^{-1})\right)\\
  &+\left(\sigma_{1}(D)\star\sigma_{-3}(|D|^{-1})\right)+\left(\sigma_1(D)\star\sigma_{-4}(|D|^{-1})\right)+\cdots\\
  &+\left(\sigma_{0}(D)\star\sigma_{-1}(|D|^{-1})\right)+\left(\sigma_{0}(D)\star\sigma_{-2}(|D|^{-1})\right)+\left(\sigma_0(D)\star\sigma_{-3}(|D|^{-1})\right)+\cdots.\\ 
   \end{align*}
   For $a=1,0$ and $b=-1,-2,-3.\cdots$, one has
   $$\sigma_a(D)\star\sigma_b(|D|^{-1})=\sum_{\alpha}\frac{1}{\alpha!}\partial_{\xi}^{\alpha}\sigma_a(D)\delta^{\alpha}\sigma_b(|D|^{-1}),$$
   and each term is of order $a-|\alpha|+b$. By collecting the terms of order $-3$ we get
   \begin{align*}
   \sigma_{-3}(D|D|^{-1})&\sim\left(\sigma_1(D)\sigma_{-4}(|D|^{-1})\right)+\left(\sigma_0(D)\sigma_{-3}(|D|^{-1})\right)\\
   &+\left(\sum_{\alpha_1+\alpha_2+\alpha_3=1}\partial_{\xi_1}^{\alpha_1}\partial_{\xi_2}^{\alpha_2}\partial_{\xi_3}^{\alpha_3}\sigma_{1}(D)\delta_1^{\alpha_1}\delta_2^{\alpha_2}\delta_3^{\alpha_3}\sigma_{-3}(|D|^{-1})\right)\\
   &+\left(\sum_{\alpha_1+\alpha_2+\alpha_3=2}\frac{1}{\alpha_1!\alpha_2!\alpha_3!}\partial_{\xi_1}^{\alpha_1}\partial_{\xi_2}^{\alpha_2}\partial_{\xi_3}^{\alpha_3}\sigma_{1}(D)\delta_1^{\alpha_1}\delta_2^{\alpha_2}\delta_3^{\alpha_3}\sigma_{-2}(|D|^{-1})\right)\\
   &+\left(\sum_{\alpha_1+\alpha_2+\alpha_3=1}\partial_{\xi_1}^{\alpha_1}\partial_{\xi_2}^{\alpha_2}\partial_{\xi_3}^{\alpha_3}\sigma_{0}(D)\delta_1^{\alpha_1}\delta_2^{\alpha_2}\delta_3^{\alpha_3}\sigma_{-2}(|D|^{-1})\right)\\
   &+\left(\sum_{\alpha_1+\alpha_2+\alpha_3=2}\frac{1}{\alpha_1!\alpha_2!\alpha_3!}\partial_{\xi_1}^{\alpha_1}\partial_{\xi_2}^{\alpha_2}\partial_{\xi_3}^{\alpha_3}\sigma_{0}(D)\delta_1^{\alpha_1}\delta_2^{\alpha_2}\delta_3^{\alpha_3}\sigma_{-1}(|D|^{-1})\right).\\
   \end{align*}
   Since $\sigma_0(D)$ has no $\xi$ dependence, the last two terms vanish and therefore we have:
    \begin{align*}
   \sigma_{-3}(D|D|^{-1})&\sim\left(\sigma_1(D)\sigma_{-4}(|D|^{-1})\right)+\left(\sigma_0(D)\sigma_{-3}(|D|^{-1})\right)\\
   &+\left(\sum_{\alpha_1+\alpha_2+\alpha_3=1}\partial_{\xi_1}^{\alpha_1}\partial_{\xi_2}^{\alpha_2}\partial_{\xi_3}^{\alpha_3}\sigma_{1}(D)\delta_1^{\alpha_1}\delta_2^{\alpha_2}\delta_3^{\alpha_3}\sigma_{-3}(|D|^{-1})\right)\\
   &+\left(\sum_{\alpha_1+\alpha_2+\alpha_3=2}\frac{1}{\alpha_1!\alpha_2!\alpha_3!}\partial_{\xi_1}^{\alpha_1}\partial_{\xi_2}^{\alpha_2}\partial_{\xi_3}^{\alpha_3}\sigma_{1}(D)\delta_1^{\alpha_1}\delta_2^{\alpha_2}\delta_3^{\alpha_3}\sigma_{-2}(|D|^{-1})\right).\\
   \end{align*}
   Finally, similar to the proof of Proposition \eqref{eta-reg}, one observes that the matrix coefficients of all terms in $\sigma_{-3}(D|D|^{-1})$ consist of either one $\gamma$ matrix or product of three $\gamma$ matrices.
  Again, by using the similar trace identities of $\gamma$ matrices identities we obtain,
$$\Wres(\frac{D}{|D|})=\tau\left(\res(D|D|^{-1})\right)=\tau\left(\int_{|\xi|=1}\tr(\sigma_{-3})d\xi\right)=0. $$
\end{proof}

     \subsection{Conformal invariance of $\eta_D(0)$}
    In this section we study the variations of $\eta_{D}(0)$ for the spectral triple $(\T^3,\mathcal{H}, D=\slashed{\partial})$. We consider the conformal variation of the Dirac operator and show that $\eta_D(0)$ will remain unchanged.
  
 Consider the commutative spectral triple $(C^{\infty}(M), L^2(M,S)_g,D_g)$ encoding the data of a closed $n$-dimensional spin Riemannian manifold with the spin Dirac operator on the space of spinors. By varying $g$ within its conformal class, we consider $\tilde{g}=k^{-2}g$ for some $k=e^{h}>0$ in $C^{\infty}(M)$. The  volume form for the perturbed metric is given by
$\dvol_{\tilde{g}}=k^{-n}\dvol_{g}$  and one has a unitary isomorphism $$U:L^2(M,S)_g\longrightarrow L^2(M,S)_{\tilde{g}}$$ by
$$U(\psi)=k^{\frac{n}{2}}\psi,$$
It can be shown that (cf. \cite{Hitchin-harmonic}) $D_{\tilde{g}}=k^{\frac{n+1}{2}}D_g k^{\frac{-n+1}{2}}$ and hence

$$U^*D_{\tilde{g}}U=k^{\frac{-n}{2}}(k^{\frac{n+1}{2}}D_gk^{\frac{-n+1}{2}})k^{\frac{n}{2}}=\sqrt{k}D_g\sqrt{k}.$$
 
The above property of the Dirac operator is usually referred to as being conformally covariant. It is a known fact that $\eta_D(0)$ for a Dirac operator on an odd dimensional manifold is a conformal invariant \cite{APS}, i.e. it is invariant under the conformal changes of the metric. In a more general context, in \cite{Rosenberg-covariant, Parker-Rosenberg} and \cite{Paycha-Rosenberg2006}, by using variational techniques it was shown that for a coformally covariant self-adjoint differential operator $A$, $\eta_A(0)$ is a conformal invariant.

     In the framework of noncommutative geometry, conformal perturbation of the  metric is implemented by changing the volume form \cite{Cohen-Connes1992}, namely, by fixing a positive element $k=e^h$ for $h^*=h$ in $C^{\infty}(\mathbb{T}^n_{\theta})$, one constructs the following positive functional
    $$\varphi_k(a)=\tau (ak^{-n}).$$
    By normalizing the above functional one obtains a state which we denote by $\varphi$.\\
    The state $\varphi$ defines an inner product 
    $$\left<a,b\right>_{\varphi}=\varphi(b^*a)\quad a,b\in C^{\infty}(\mathbb{T}^n_{\theta})$$
   and hence one obtains a Hilbert space $\mathcal{H}_{\varphi}$ by GNS construction. The algebra $C^{\infty}(\mathbb{T}^n_{\theta})$ acts unitarilly on $\mathcal{H}_{\varphi}$ by left regular representation and the right multiplication operator $R_{k^{n/2}}$ extends to a unitary map $U_0:\mathcal{H}_{\tau}\longrightarrow\mathcal{H}_{\varphi}$. In fact one has,
   $$\left<U_0a,U_0b\right>_{\varphi}=\varphi(k^{n/2}b^*ak^{n/2})=\tau(k^{n/2}b^*ak^{n/2}k^{-n})=\tau(b^*a)=\left<a,b\right>_{\tau}.$$ 
   
   We put $\mathcal{\tilde{H}}=H_{\varphi}\otimes\mathbb{C}^N$, the action of $C^{\infty}(\mathbb{T}^n_{\theta})$ on $\mathcal{\tilde{H}}$ is given by 
$$a\longrightarrow a\otimes1$$
    and the map 
    $$U=U_0\otimes I:\mathcal{H}\longrightarrow\mathcal{\tilde{H}}$$
is a unitary equivalence between the two Hilbert spaces. \\
Now we consider the operator $\tilde{D}=R_{k^{\frac{n+1}{2}}}DR_{k^{\frac{-n+1}{2}}}$.
\begin{prop}
$(C^{\infty}(\mathbb{T}^n_{\theta}),\mathcal{\tilde{H}},\tilde{D})$ and $(C^{\infty}(\mathbb{T}^n_{\theta}),\mathcal{H},R_{\sqrt{k}}DR_{\sqrt{k}})$ are spectral triples, and the map $U$ is a unitary equivalence between them.
\end{prop}
  \begin{proof}
  Note that left multiplication by an element $a\in C^{\infty}(\mathbb{T}^n_{\theta}) $ commutes with right multiplication operators $R_{k^2}$, $R_{k^{-1}}$ and $R_{\sqrt{k}}$ . Also the two norms coming from $<,>_{\varphi}$ and $<,>_{\tau}$ are equivalent. So both commutators $[a,\tilde{D}]$ and $[a,\sqrt{k}D\sqrt{k}]$ are bounded. The unitary equivalence easily follows from definition of $U$ and $\tilde{D}$. 
  \end{proof}
  In next step, we convert the right multiplications in $R_{\sqrt{k}}DR_{\sqrt{k}}$ to left multiplication using the real structure on noncommutative tori (see example \ref{real-tori}) .
  It is easily seen that 
  $$J_0R_{\sqrt{k}}\partial_{\mu}R_{\sqrt{k}}J_0=-\sqrt{k}\partial_{\mu}\sqrt{k},$$
  and 
  $$JR_{\sqrt{k}}DR_{\sqrt{k}}J=\sqrt{k}D\sqrt{k}.$$
  
   Since $\sqrt{k}D\sqrt{k}$ is iso-spectral  to $R_{\sqrt{k}}DR_{\sqrt{k}}$ (being intertwined by $J$), the following definition is reasonable.
  \begin{defi}
  The conformal perturbation of the Dirac operator $D$ is the 1-parameter family $(C^{\infty}(\mathbb{T}^n_{\theta}),\mathcal{H},D_t)$, where 
  $$D_t=e^{\frac{th}{2}}De^{\frac{th}{2}},~~~h=h^*\in C^{\infty}(\mathbb{T}^n_{\theta}).$$
    \end{defi}
 We need the following formula for variation of eta function.
 \begin{lem}\label{eta-var}
 Let $\{D_t\}$ be a smooth 1-parameter family of invertible self-adjoint elliptic operators of order $d$, then
 $$\frac{d}{dt}\eta_{D_t}(z)=-z\TR(\dot{D_t}(D^2_t)^{-(\frac{z+1}{2})}).$$
 \end{lem}
 \begin{proof}
 For $k>0$ odd, $\eta_{D^k}(z)=\eta_D(kz)$, hence we can replace $D$ by $D^k$ for $k$ large enough.\\
 For $d$ large enough, $(D-\lambda)^{-1}$ is trace class, so one can write
 $$\eta_{D_t}(z)=\frac{1}{2\pi i}\int_{\Gamma_1}\lambda^{-z}\Tr(D_t-\lambda)^{-1}d\lambda-\int_{\Gamma_2}(-\lambda)^{-z}\Tr(D_t-\lambda)^{-1}d\lambda,$$
 where $\Gamma_1$ and $\Gamma_2$ are appropriate contours around positive and negative eigenvalues of $D$.
 
 Now, 
 $$\frac{d}{dt}(D_t-\lambda)^{-1}=-(D_t-\lambda)^{-1}\dot{D_t}(D_t-\lambda)^{-1},$$
 so
 $$\Tr(\frac{d}{dt}(D_t-\lambda)^{-1})=-\Tr(\dot{D_t}(D_t-\lambda)^{-2}),$$
 therefore
 $$\frac{d}{dt}\eta_{D_t}(z)=\frac{1}{2\pi i}\Tr(\dot{D_t}(\int_{\Gamma_1}-\lambda^{-s}\Tr(D_t-\lambda)^{-2}d\lambda
 +\int_{\Gamma_2}(-\lambda)^{-z}\Tr(D_t-\lambda)^{-2}d\lambda)).$$
 Now, integration by parts in both integrals and the formula $\frac{d}{d\lambda}(D_t-\lambda)^{-1}=(D_t-\lambda)^{-2}$ gives the result.
 
 \end{proof} 
\begin{rmk}
\normalfont
  By applying the above lemma to the Dirac operator $D$ on noncommutative 3-torus we get,
  \begin{align*}
   \partial\eta_D(0)=\left[\left.\frac{d}{dt}\right|_{t=0}\eta_{D_t}(z)\right]_{z=0}&=\left[-z\TR(\partial D|D|^{-1}(D ^2)^{\frac{-z}{2}})\right]_{z=0}\\
   =-\Wres(\partial D|D|^{-1})&.
  \end{align*}
 
  By the properties of Wodzicki residue, one sees that $\eta(0)$ is constant under smoothing perturbations, so if $\ker(D)\neq0$, on can replace $D$ by $D+\Pi$ where $\Pi$ is the projection on the finite dimensional kernel of $D$ and hence the result of above lemma still makes sense for an operator with nontrivial kernel.
\end{rmk}
  
  \begin{prop} Consider the spectral triple $(\T^3,\mathcal{H},D=\slashed{\partial})$, then $\eta_D(0)$ is invariant under the conformal perturbations.
  \end{prop}
  \begin{proof}

  Consider the 1-parameter family of operators 
  $$D_t=e^{\frac{th}{2}}De^{\frac{th}{2}},~~~h=h^*\in\T^3.$$
  One computes 
  $$\dot{D_t}=\frac{1}{2}(hD_t+D_th),$$
  and 
  $$\partial D=\frac{1}{2}(hD+Dh).$$
 Therefore by  trace property of noncommutative residue we get
  $$\partial\eta(D,0)=-\Wres(hD|D|^{-1}),$$
 and it is seen that $$\Wres(hD|D|^{-1})=\tau\left( \tr( \res(hD|D|^{-1})\right)=\tau\left(\tr(h~\res(D|D|^{-1})\right).$$
 Now by looking at the symbols 
 \begin{align*}
 \sigma(D)&\sim\xi_{\mu}\otimes\gamma^{\mu},\\
 \sigma(|D|)&\sim |\xi|^2\otimes I,
 \end{align*}
 we get,
$$\sigma(D|D|^{-1})\sim\frac{\xi_{\mu}}{|\xi|}\otimes\gamma^{\mu}.$$
 Therefore, there is no homogeneous term of negative order in the symbol of $\sigma(D|D|^{-1})$ and hence,
  $$\Wres(hD|D|^{-1})=0.$$
  \end{proof}
  
  \subsection{Spectral flow and odd local index formula}
The value of eta function at zero is intimately related to another spectral quantity called the spectral flow. To motivate this relation, consider a family $A_t$ of $N\times N$ Hermitian matrices. The spectral flow of the family $A_t$ is defined as the net number of the eigenvalues of $A_t$ passing zero. It is easily seen that the difference of the signatures of the end points is related to the spectral flow of the family by,
$$\eta_{A_1}(0)-\eta_{A_0}(0)=2\SF(A_t).$$  
The spectral flow of a family of self adjoint elliptic operators $A_t$ on manifolds can also be defined (\cite{APS}) but the above equality holds along with a correction term (see \cite{Paycha-Cardona} for a proof).
\begin{prop}\label{eta-flow}
Let $A_t$ be a self-adjoint family of  elliptic operators of order $a$ with $A_0$ and $A_1$ invertible, then 
\begin{equation} 
\eta_{A_1}(0)-\eta_{A_0}(0)=2\SF(A_t)-\frac{1}{a}\int_0^1\frac{d}{dt}\eta_{A_t}(0)dt.  
\end{equation} 
\end{prop}
\qed

 It's a remarkable fact due to Getzler (\cite{Getzler-flow}) that the spectral flow of the family of operators $D_t=D+tg^{-1}[D,g]$ interpolating $D$ and $g^{-1}Dg$ for the Dirac operator $D$ on an odd dimensional spin manifold and $g\in C^{\infty}(M,GL(N))$ in fact gives the index of $PgP$ where $P=\frac{1+F}{2}$ and $F=D|D|^{-1}$ is the sign operator. Therefore one obtains a pairing between the odd K-theory and K-homology:
 \begin{equation}\label{flow-index}
\SF(D_t)=\index(PgP).
 \end{equation}
 
 The above equality  has been generalized to the  framework of non commutative index theory \cite{Carey-Phillips98, Carey-Phillips}.
 Consider a spectral triple $(\mathcal{A},\mathcal{H},D)$ with sufficiently nice regularity properties and let $u\in\mathcal{A}$ be a unitary element  and $[u]\in K^1(A)$ the corresponding class in K-theory. The sign operator $F=\frac{D}{|D|}$ gives the projection $P=\frac{1+F}{2}$ and $PuP$ is a Fredholm operator. For the family $D_t=D+tu^*[D,u]$ interpolating between $D$ and $u^*Du$ one can show that(see \cite{Carey-Phillips98} and \cite{Carey-Phillips}):
 \begin{equation}
 \SF(D_t)=\index(PuP).
\end{equation}   
 The odd local index formula of Connes-Moscovici in turn expresses the right side of the above equality as a pairing between the periodic cyclic cohomolgy and K-theory given by a local formula in the sense of non commutative geometry \cite{CM-local-index}. 
 Note that for the spectral triple $(\T^3,\mathcal{H},D=\slashed{\partial}$), by Proposition (\ref{eta-coupled}), the eta function $\eta_{D_t}(0)$ over the family of Dirac operators $D_t=D+tu^*[D,u]$ makes sense and with a minor modification of the proof of Proposition (\ref{eta-flow}) in the commutative case we get,
 \begin{equation}
  \eta_{u^*Du}(0)-\eta_{D}(0)=2\SF(D_t)-\int_0^1\frac{d}{dt}\eta_{D_t}(0)dt.
 \end{equation}
 Since $u^*Du$ and $D$ have the same spectrum, we have $\eta_{u^*Du}(0)=\eta_{D}(0)$ and therefore we obtain obtain the following integrated local formula for the index:
  \begin{equation}
 \index(PuP)=\frac{1}{2}\int_0^1\frac{d}{dt}\eta_{D_t}(0)dt.
  \end{equation}

\section{Bosonic functional determinant and conformal anomaly}

Consider a positive elliptic differential operator $A$ on a closed manifold $M$, the corresponding bosonic action is defined by
\begin{align*}
        S_{\bos}(\phi)=(\phi,A\psi),
\end{align*} 
where $\phi:M\rightarrow\mathbb{C}^d$ is a bosonic field living on $M$.
Upon quantization, one constructs the partition function 
\begin{align*}
\mathcal{Z}_{\bos}=\int e^{\frac{-1}{2}S_{bos}(\phi)}[\mathcal{D}\phi],
\end{align*}
where $[\mathcal{D}\phi]$ is a formal measure on the configuration of bosonic fields. The corresponding effective action is given by 
\begin{align*}
\mathcal{W}=\log\mathcal{Z}.
\end{align*}
One should think of $\mathcal{Z}_{\bos}$ formally as $(\det A)^{-\frac{1}{2}}$. Through zeta regularization scheme (see e.g. \cite{QFT-analytic})
one defines,
\begin{equation}
\mathcal{Z}:=e^{\frac{1}{2}\zeta_A'(0)}
\end{equation}
and hence,
\begin{equation}
\mathcal{W}=\frac{1}{2}\zeta_A'(0),
\end{equation}
where the spectral zeta function is defined by
$$\zeta_A(z)=\TR(A^{-z}),$$
and we have dropped the dependence on the spectral cut $L_{\theta}$ in definition of the complex power. Note that since $A$ is a differential operator,  by Proposition (\ref{Laurent exp}), $\zeta_A(z)$ is regular at $z=0$ and $\zeta_A'(0)$ makes sense.

Let $D$ be the Dirac operator on a closed odd dimensional spin manifold, it's a known fact that $\zeta_{|D|}'(0)$ is a conformal invariant (see \cite{Rosenberg-covariant} and \cite{Parker-Rosenberg}).  We prove the analogue of this result for the Dirac operator on noncommutative n-torus $(C^{\infty}(\mathbb{T}^n_{\theta}),\mathcal{H},D=\slashed{\partial})$ in odd dimensions.

 In even dimensions the conformal variation is not zero and hence conformal quantum anomaly exists. In dimension two, the Polyakov anomaly formula \cite{Polyakov} gives a local expression for conformal anomaly of the Laplacian on a Riemann surface. Following \cite{Connes-Moscovici2014} we derive an analogue of this formula for noncommutative 2-torus using the formalism of canonical trace.

For proving the conformal invariance of $\zeta_{|D|}'(0)$ on $C^{\infty}(\mathbb{T}^{2p+1}_{\theta})$ we will need to carefully analyze the coefficients in short time asymptotic of  the trace of heat operator $e^{-tD^2}$.

\begin{lem}\label{heatasymp}
Consider the spectral triple $(C^{\infty}(\mathbb{T}^n_{\theta}),\mathcal{H},D)$.  One has the following asymptotic expansion
$$\Tr(e^{-tD^2})\sim\sum_{i=0}^{\infty}a_i(D^2)t^{\frac{i-n}{2}}\quad t\rightarrow0,$$
where coefficients $a_i$ are computed by integrating local terms. 
\end{lem}
\begin{proof}
 We have,
 $$\Tr(e^{-tD^2})=\tau(\int_{\mathbb{R}^n}\tr\sigma(e^{-tD^2})(\xi)d\xi).$$
The operator $e^{-tD^2}$ is again a pseudodifferential operator of order $-\infty$ whose symbol can be computed using the symbol product rules. 
 For any $\lambda$ in the resolvent of $D^2$, $(\lambda-D^2)^{-1}$ is a pseudodifferential operator of order $-2$ with symbol $r(\xi,\lambda)$ which can be written as its homogeneous parts $r=r_0+r_1+\cdots$, with $r_k(\xi,\lambda)$ homogeneous of order $-k-2$ in $(\xi,\lambda)$ i.e.
\begin{equation}
r_k(t\xi,t^2\lambda)=t^{-2-k}r_k(\xi,\lambda)\quad \forall t>0
\end{equation}  
The operator $e^{-tD^2}$ is also a pseudodifferential operator with the symbol $e_0+e_1+e_2+\cdots$ where $e_n\in S^{-\infty}$ and are defined by
\begin{equation}
e_n(t,\xi)=\frac{1}{2\pi i}\int_\gamma e^{-t\lambda}r_n(\xi, \lambda)d\lambda\quad t>0
\end{equation}

\begin{align*}
 \int_{\mathbb{R}^n} \sigma(e^{-tD^2})(t,\xi)d\xi\\
&=\sum_{i} \int  e_i(t,\xi)d\xi\\
\end{align*} 
The integrands $\int_{\mathbb{R}^n}  e_i(t,\xi)d\xi$ are homogeneous in $\xi$. 
\begin{align*}
\int_{\mathbb{R}^n} e_n(t,\xi) d\xi&= \frac{1}{2\pi i}\int_{\mathbb{R}^n}  \int_\gamma e^{-t\lambda}r_n(x,\xi, \lambda)d\lambda d\xi\\
&=\frac{1}{2\pi i}\int_{\mathbb{R}^n}  \int_\gamma e^{-\lambda'}r_n(t^{-1/2}\xi', t^{-1}\lambda')\frac{d\lambda'}{t} \frac{d\xi'}{t^{n/2}}\\
&=t^{\left(\frac{i-n}{2}\right)}\frac{1}{2\pi i}\int_{\mathbb{R}^n} \int_\gamma e^{-\lambda}r_n(\xi, \lambda)d\lambda d\xi\\
\end{align*} 
In identity second identity we have used the change of variable $t\lambda=\lambda'$ and $t^{1/2}\xi=\xi'$.

 Hence we have
\begin{equation}
\int_{\mathbb{R}^n}\sigma(e^{-tD^2})(\xi)d\xi\sim\sum_{i=0}\beta_i t^{\frac{i-n}{2}},
\end{equation}
Where 
\begin{equation}
\beta_i(x)=\frac{1}{2\pi i}\int_{\mathbb{R}^n} \int_\gamma e^{-\lambda}r_i(\xi, \lambda)d\lambda d\xi
\end{equation}
and hence
$$\Tr(e^{-tD^2})\sim\sum_{i=0}^{\infty}t^{\frac{i-n}{2}}a_i(D^2),~~t\rightarrow0,$$
Where $a_i(D^2)=\tau(\tr(\beta_i))$
\end{proof}
\begin{rmk}
\normalfont
\normalfont
The above asymptotic expansion holds for any self-adjoint elliptic differnetial operator $A\in\Psi_{cl}^{*}(C^{\infty}(\mathbb{T}^n_{\theta}))$. In fact, one has
$$\Tr(e^{-tA^2})\sim\sum_{i=0}^{\infty}a_i(A^2)t^{\frac{i-n}{\alpha}}~~~t\rightarrow0,$$
where $\alpha=ord(A)$.
\end{rmk}

\begin{rmk}
\normalfont
One also has the following asymptotic expansion
$$\Tr(ae^{-tD^2})\sim\sum_{i=0}^{\infty}a_i(D^2,a)t^{\frac{i-n}{2}}~~~t\rightarrow0,$$
Where $a\in C^{\infty}(\mathbb{T}^n_{\theta})$ is considered as a multiplication operator and
$$a_i(D^2,a)=\tau(\tr(a\beta_i))$$
in above computations.
\end{rmk}

Next, we have the following variational result.
\begin{lem}\label{zetavar}
Let $A_t$ be a 1-parameter family of positive elliptic differential operators of fixed order $\alpha$ on noncommutaive n-torus, then
$$\frac{d}{dt}\zeta_{A_t}(z)=-z\TR(\dot{A_t}(A_t)^{-z-1}).$$
\end{lem}
\begin{proof}
By using the contour integral formula for the complex powers we have,
\begin{align*}
\frac{d}{dt}\TR(A_t^{-z})=\frac{1}{2\pi i}\int_{\Gamma}\lambda^{-z}\frac{d}{dt}\TR(\lambda-A_t)^{-1}d\lambda\\=\frac{1}{2\pi i}\int_{\Gamma}\lambda^{-z}\TR(\dot{A_t(}\lambda-A_t)^{-2})d\lambda\\
=\TR\left(\dot{A_t}\frac{1}{2\pi i}\int_{\Gamma}\lambda^{-z}(\lambda-A_t)^{-2}d\lambda\right).\\
\end{align*} 
Now using integration by parts formula gives the result.
\end{proof}
\begin{rmk}
\normalfont
 Note that by  above lemma we have:
\begin{equation}
\frac{d}{dt}\zeta_{A_t}(0)=\left[-z\TR(\dot{A_t}A^{-1}A^{-z})\right]_{z=0}=-\Wres(\dot{A_t}A^{-1}),
\end{equation} 
and therefore, the value of zeta at zero is constant under smoothing perturbations and hence the above result makes sense for non invertible operators.
\end{rmk}

We also need the following result relating the constant term in Laurent expansion of $\TR(AQ^{-z})$ around $z=0$ to the constant term in the asymptotic expansion  $\Tr(Ae^{-tQ})$ at $t=0$. Let $\f.p.\left.\TR(AQ^{-z})\right|_{z=0}$ denote the constant term in the Laurent expansion and also $\f.p.\left.\Tr(Ae^{-tQ})\right|_{t=0}$ denote the constant term in the heat trace asymptotic near zero. We refer the reader to \cite{Berlin-Getzler} for a proof of the following result in a more general form in commutative case.

\begin{lem}\label{BGVlemma} 
Consider the differential operators $A,Q\in\Psi_{cl}^{*}(C^{\infty}(\mathbb{T}^n_{\theta}))$, where $Q$ is positive with positive order $q$. Then,
\begin{equation}
\f.p.\left.\TR(AQ^{-z})\right|_{z=0}=\f.p.\left.\Tr(Ae^{-tQ})\right|_{t=0}.
\end{equation}
\end{lem}
\qed
\begin{prop}
Consider the spectral triple $(C^{\infty}(\mathbb{T}^{2p+1}_{\theta}),\mathcal{H},D=\slashed{\partial})$. The value $\zeta'_{|D|}(0)$ is a conformal invariant.
\end{prop}
\begin{proof}
It is enough to show that  $\zeta'_{\Delta}(0)$ for $\Delta=D^2$  is  conformally invariant. By the Lemma (\ref{zetavar}) we have,
$$\partial\zeta_{\Delta}(z)=-z\TR(\partial \Delta\cdot \Delta^{-1}\Delta^{-z}),$$
so we have the following Laurent expansion around $z=0$:
$$\partial\zeta_{\Delta}(z)=-z\left(a_{-1}\frac{1}{z}+a_0+a_1z+\cdots\right).$$
Therefore 
\begin{equation}\label{zetavar2}
\partial\zeta_{\Delta}'(0)=\frac{d}{dz}\left.\partial\zeta_{\Delta}(z)\right|_{z=0}=-a_0=-\tau\left(\int\!\!\!\!\!\!-\sigma(\partial\Delta\Delta^{-1})-\frac{1}{2}res(\partial\Delta\Delta^{-1}\log\Delta)\right).
\end{equation}

Consider the conformal perturbation of the Dirac operator,
$$D_t=e^{\frac{th}{2}}De^{\frac{th}{2}}~~~h=h^*\in\mathcal{A}_{\theta}.$$
An easy computation gives
$$\partial\Delta=\left.\frac{d}{dt}D_t^2\right|_{t=0}=\frac{h}{2}D^2+DhD+\frac{1}{2}D^2h,$$
and therefore by inserting the above identity into (\ref{zetavar2}) we get
\begin{equation}
\partial\zeta_{\Delta}'(0)=-2\tau\left(\int\!\!\!\!\!\!-\sigma(h)-\frac{1}{2}\res(h\log\Delta)\right)=-2\f.p.\left.\TR(h\Delta^{-z})\right|_{z=0}.
\end{equation}

 Now using the Lemma (\ref{BGVlemma}) we have 
\begin{equation}\label{zetavar3}
\partial\zeta_{\Delta}'(0)=-2\f.p.\left.\Tr(h e^{-t\Delta})\right|_{t=0}=-2a_{2p+1}(h,\Delta)=-2\tau(\tr(h\beta_{2p+1})).
\end{equation}
By examining the proof of Lemma (\ref{heatasymp}), we see that $a_n=0$ for odd $n$ since the integrand involved is an odd function, hence the result is obtained.
\end{proof}
In the following we give an analogue of Polyakov anomaly formula \cite{Polyakov} for the spactral triple $(C^{\infty}(\mathbb{T}^2_{\theta}),\mathcal{H},\slashed{\partial})$. 
Note that although $\log\det(\Delta)=-\zeta_{\Delta}'(0)$ is not local\footnote{Roughly, it means that it can not be written as an integral of finite number of homogenous terms of the symbol of $\Delta$. } (see Proposition (\ref{Laurent exp})),  the difference between $\log\det_{\Delta_h}$ of the conformally perturbed Laplacian and $\log\det_{\Delta}$ can be given by a local formula. This is of course an example of  the local nature of  anomalies in quantum field theory. Here we only express this difference as an integrated anomaly and refer the reader to \cite{Connes-Moscovici2014} for further computations and  interpretation of the formula.

\begin{prop}{\bf( A conformal anomaly formula)}  
Consider the spectral triple $(C^{\infty}(\mathbb{T}^2_{\theta}),\mathcal{H},D=\slashed{\partial})$, $\Delta=D^2$ and $\Delta_h=D_h^2$ where $D=e^{\frac{h}{2}}De^{\frac{h}{2}}$. The difference between $\log\det_{\Delta_h}$ of the conformally perturbed Laplacian and $\log\det_{\Delta}$ can be given by a local formula:
\begin{equation}\label{logdetpolyakov}
\log\det(\Delta_h)-\log\det(\Delta)=-\left(\zeta_{\Delta_h}'(0)-\zeta_{\Delta}'(0)\right)=-\int_0^1\frac{d}{dt}\zeta_{\Delta_t}'(0)dt,
\end{equation}
where $\Delta_t=D_t^2$,  $D=e^{\frac{th}{2}}De^{\frac{th}{2}}$.
\end{prop}
 \begin{proof}
 The proof follows from the Lemma (\ref{zetavar}) and fundamental theorem of calculus along the family $\Delta_t$.
 \end{proof}
 \begin{rmk}
 \normalfont
 Note that, by using the similar argument used in derivation of  the equations \eqref{zetavar2} and \eqref{zetavar3} we see  that the integrand in equation 
 \eqref{logdetpolyakov}  is given by
 \begin{align*}
  &\frac{d}{dt}\zeta_{\Delta_t}'(0)=-\tau\left(\int\!\!\!\!\!\!-\sigma(\frac{d}{dt}\Delta_t\Delta_{t}^{-1})-\frac{1}{2}\res(\frac{d}{dt}\Delta_t\Delta_{t}^{-1}\log\Delta_{t})\right)=\\
 & -2a_{2}(h,\Delta_t)=-2\tau\left(\tr(h\beta_{2}(\Delta_t))\right).
 \end{align*}
of course, computing the density $\tr(\beta_{2}(\Delta_t))$ requires considerable amount of computations (see \cite{Connes-Moscovici2014} and \cite{Farzad-Masoud2013}).
 \end{rmk}
 \section{Fermionic functional determinant and  the induced Chern-Simons term}
       Consider a closed manifold $M$ and the classical fermionic action defined by
        $$S_{\fer}(\bar{\psi},\psi)=(\bar{\psi},D\psi),$$
    where $\bar{\psi}$ and $\psi $ are fermion fields on $M$ and $D=i\gamma^{\mu}\nabla_{\mu}$ is the Dirac operator .
    The fermionic partition function is given by
    $$\mathcal{Z}_{\fer}=\int e^{-S_{\fer}(\bar{\psi},\psi)}[\mathcal{D}\bar{\psi}][\mathcal{D}\psi],$$
    where $[\mathcal{D}\bar{\psi}]$ and $[\mathcal{D}\psi]$ are formal measures on the space of fermions. The partition function $\mathcal{Z}$ should be thought as the $\det(D)$.
    The one-loop fermionic effective action is defined as 
    $$\mathcal{W}:=\log(\mathcal{Z})=\log\det(D).$$ 
    Similar to bosonic case, 
       after choosing a spectral cut $L_{\phi}$ \footnote{ There are  two choices of spectral cut $L_{\phi}$ for the Dirac operator: in upper or lower half plane}, by definition we have (see \eqref{Qz}):
   $$\zeta_D(z)=TR(D_{\phi}^{-z}),$$
   and hence one can define
  \begin{align}
      \mathcal{W}=\log\det(D):=-\zeta_D'(0).
  \end{align}
   Note that due to choice of a spectral cut, there is an ambiguity involved in above definition of one-loop effective action. Let $\zeta_D ^{\uparrow}(z)$ and $\zeta_D ^{\downarrow}(z)$ be the spectral zeta functions corresponding to a choice of $\phi$ in upper and lower half plane respectively. a quick computation shows that (cf. \cite{Ponge-assym})
   \begin{align*}
   \zeta_D ^{\uparrow}(z)-\zeta_D ^{\downarrow}(z)=(1-e^{-i\pi z})\zeta_D ^{\uparrow}(z)-(1-e^{-i\pi z})\eta_{D}(z),
\end{align*}    
and the measure of ambiguity in the effective action is given by
\begin{align}
\zeta_D ^{\uparrow}(0) '-\zeta_D ^{\downarrow}(0)'=i\pi\zeta_D ^{\uparrow}(0)+i\pi\eta_D(0).
\end{align}
In $2p+1$ dimensions $\zeta_D ^{\uparrow}(0)=\zeta_D ^{\downarrow}(0)=0$ and hence
\begin{align*}
\zeta_D ^{\uparrow}(0) '-\zeta_D ^{\downarrow}(0)'=i\pi\eta_D(0).
\end{align*}
Therefore the ambiguity is given by the non local quantity $\eta_D(0)$ which also depends on the gauge field coupled to the Dirac operator. The measure of this dependence
can be given by a local formula which in physics literature is referred to as the induced Chern-Simons term generated by the coupling of a massless fermion to a classical gauge field (cf. e.g. \cite{QFT-analytic}).
    
   Here, we give an analogue of this local term for the coupled Dirac operator on noncommutative 3-torus. 
    We consider the  operator  $D=\slashed{\partial}+\slashed{A}$ on $\T^3$ and compute the variation of the eta invariant $\eta_{D}(0)$ with respect to the vector potential. 
    
    First we state the following lemma. The proof in commutative case also works in noncommutative setting with minor changes and we will not reproduce it here (see \cite{Paycha-Rosenberg2006}).
\begin{lem}\label{res-asym}
 In the asymptotic expansion of the heat kernel for a positive elliptic differential operator $A\in\Psi_{cl}^{*}(C^{\infty}(\mathbb{T}^n_{\theta}))$ of $ord(A)=\alpha$,
$$\int_{\mathbb{R}^n}\sigma(e^{-tA})(\xi)d\xi\sim\sum_{i=0}^{\infty}\beta_i t^{\frac{i-n}{\alpha}},$$
one has
$$\res(A^{-k})=\frac{\alpha}{(k-1)!}\beta_{n-\alpha k},$$
where $k\in\mathbb{Z}^{+}$ and $\beta_{n-\alpha k}=0$ if $\alpha k\notin\mathbb{Z}$.
\end{lem}
\qed

Consider the family $D_t=\slashed{\partial}+\slashed{A}_t$ on $\T^3$. By performing the variation with respect to $A_{\mu}$ we have,
\begin{equation}
\partial D=\gamma^{\mu}\partial A_{\mu}.
\end{equation}

For the family of Dirac operators $\{D_t\}$, a proof similar to the proof of Proposition (\ref{eta-coupled}) shows that $\eta_{D_t}(z)$ is regular at $z=0$ and therefore along this family, $\eta_{D_t}(0)$ makes sense. By using Lemma (\ref{eta-var}) we obtain: 
\begin{equation}
\partial\eta(0)=\Wres\left(\partial D|D|^{-1}\right)=\Wres\left(\gamma^{\mu}\partial A_{\mu}|D|^{-1}\right).
\end{equation}

Now by using Lemma (\ref{res-asym}) and the fact that $\gamma^{\mu}\partial A_{\mu}$ is a zero order operator it follows that the variation of the eta invariant with respect to the gauge field $A_{\mu}$ is given by the following local formula,
\begin{equation}
\partial\eta_{D}(0)=\Wres\left(\gamma^{\mu}\partial A_{\mu}|D|^{-1}\right)=\tau \left(\tr\left(\gamma^{\mu}\partial A_{\mu}\beta_{2}\right)\right).
\end{equation}

\def\polhk#1{\setbox0=\hbox{#1}{\ooalign{\hidewidth
  \lower1.5ex\hbox{`}\hidewidth\crcr\unhbox0}}} \def\cprime{$'$}

\end{document}